\documentclass{amsart} 
\usepackage{graphicx}
\usepackage{hyperref}
\usepackage{amsmath, amsthm, amssymb, amscd}

\usepackage{mathrsfs}
\usepackage{mathtools}
\usepackage{enumerate}
\usepackage{hyperref}
\usepackage{verbatim}
\usepackage{centernot}
\usepackage{subcaption}
\usepackage{cite}
\usepackage{color}
\usepackage{esint}
\usepackage{tikz-cd}
\usetikzlibrary{shapes.geometric}
\usepackage[shortlabels]{enumitem}
 
\usepackage{bm}
\usepackage{bbm, dsfont}
\definecolor{mycolor}{rgb}{0.0, 0.75, 1.0}

\makeatletter
\newcommand{\labitem}[2]{%
\def\@itemlabel{\textbf{#1}}
\item
\def\@currentlabel{#1}\label{#2}}
\makeatother

\title[Sub-Gaussian heat kernel estimates via Analysis on Metric Spaces]{An Approach to Sub-Gaussian Heat Kernel Estimates via Analysis on Metric Spaces}
\author{Riku Anttila}
\address[Riku Anttila]{Department of Mathematics and Statistics, University of Jyväskylä, P.O. Box 35, FI-40014 Jyväskylä, Finland}
\email{riku.t.anttila@jyu.fi}


\date{\today}

\usepackage[shortlabels]{enumitem}
 
\newtheorem{theorem}[equation]{Theorem}
\newtheorem{lemma}[equation]{Lemma}

\newtheorem{proposition}[equation]{Proposition}
\newtheorem{corollary}[equation]{Corollary}

\newtheorem{conjecture}[equation]{Conjecture}

\newtheorem*{theorem*}{Theorem}
\newtheorem*{question*}{Question}
\newtheorem*{mainQ*}{Main question}

\numberwithin{equation}{section}

\theoremstyle{definition}
\newtheorem{definition}[equation]{Definition}
\newtheorem{example}[equation]{Example}

\theoremstyle{remark}
\newtheorem{remark}[equation]{Remark}
    \newcommand*{\N}{\mathbb{N}}
    \newcommand*{\Z}{\mathbb{Z}}
    
    \newcommand*{\R}{\mathbb{R}}
    
        \DeclarePairedDelimiter\Span{\langle}{\rangle}
        
        \DeclareMathOperator{\diam}{diam}

        \DeclareMathOperator{\supp}{supp}
        \DeclareMathOperator{\Lip}{Lip}

        \DeclarePairedDelimiter\abs{\lvert}{\rvert}
        \DeclarePairedDelimiter\norm{\lVert}{\rVert}
        \DeclareMathOperator*{\esssup}{ess\,sup}
        
        \def\vint_#1{\mathchoice%
          {\mathop{\kern 0.2em\vrule width 0.6em height 0.69678ex depth -0.58065ex
                  \kern -0.8em \intop}\nolimits_{\kern -0.4em#1}}%
          {\mathop{\kern 0.1em\vrule width 0.5em height 0.69678ex depth -0.60387ex
                  \kern -0.6em \intop}\nolimits_{#1}}%
          {\mathop{\kern 0.1em\vrule width 0.5em height 0.69678ex
              depth -0.60387ex
                  \kern -0.6em \intop}\nolimits_{#1}}%
          {\mathop{\kern 0.1em\vrule width 0.5em height 0.69678ex depth -0.60387ex
                  \kern -0.6em \intop}\nolimits_{#1}}}
\def\vintslides_#1{\mathchoice%
          {\mathop{\kern 0.1em\vrule width 0.5em height 0.697ex depth -0.581ex
                  \kern -0.6em \intop}\nolimits_{\kern -0.4em#1}}%
          {\mathop{\kern 0.1em\vrule width 0.3em height 0.697ex depth -0.604ex
                  \kern -0.4em \intop}\nolimits_{#1}}%
          {\mathop{\kern 0.1em\vrule width 0.3em height 0.697ex depth -0.604ex
                  \kern -0.4em \intop}\nolimits_{#1}}%
          {\mathop{\kern 0.1em\vrule width 0.3em height 0.697ex depth -0.604ex
                  \kern -0.4em \intop}\nolimits_{#1}}}

\newcommand{\kint}{\vint}

\newcommand{\aveint}[2]{\mathchoice%
          {\mathop{\kern 0.2em\vrule width 0.6em height 0.69678ex depth -0.58065ex
                  \kern -0.8em \intop}\nolimits_{\kern -0.45em#1}^{#2}}%
          {\mathop{\kern 0.1em\vrule width 0.5em height 0.69678ex depth -0.60387ex
                  \kern -0.6em \intop}\nolimits_{#1}^{#2}}%
          {\mathop{\kern 0.1em\vrule width 0.5em height 0.69678ex depth -0.60387ex
                  \kern -0.6em \intop}\nolimits_{#1}^{#2}}%
          {\mathop{\kern 0.1em\vrule width 0.5em height 0.69678ex depth -0.60387ex
                  \kern -0.6em \intop}\nolimits_{#1}^{#2}}}
\usepackage{systeme}
\makeatletter
\let\c@equation\c@figure
\makeatother

\begin{document}

\begin{abstract}
    In this work, we establish a new characterization of sub-Gaussian heat kernel estimates for strongly local regular Dirichlet forms on metric measure spaces.
    Our formulation is based on the newly introduced \emph{cutoff energy condition}, which offers a simpler and more transparent alternative for earlier technical energy inequalities, in particular the cutoff Sobolev inequality.
    The main idea of our approach is to reinterpret the cutoff Sobolev inequality as a Poincar\'e type inequality, and analyze it using Haj{\l}asz--Koskela techniques from analysis on metric spaces. Applications of the new characterization are also discussed.
\end{abstract}

\maketitle
\section{Introduction}
\subsection{Background}

Let $(X,d,\mu)$ be a metric measure space where $(X,d)$ is a complete, locally compact, and separable metric space. We assume that $\mu$ is a Radon measure on $(X,d)$ with full topological support.
To simplify some statements on the known results, we assume that $(X,d)$ is geodesic.

The objective of this work is to deepen our understanding of the following question: When does a given strongly local Dirichlet form $(\mathcal{E},\mathcal{F})$ on $L^2(X,\mu)$ satisfy the \emph{sub-Gaussian heat kernel estimates},
\begin{align}\label{eq:HK}
    \tag{$\textup{HKE}(\beta)$}
    \frac{c}{\mu(B(x,t^{1/\beta}))}& \exp\left( -\left(\frac{d(x,y)^\beta}{ct}\right)^{\frac{1}{\beta - 1}} \right)\\
    & \leq p_{t}(x,y)
    \leq \frac{C}{\mu(B(x,t^{1/\beta}))} \exp\left( -\left(\frac{d(x,y)^\beta}{Ct}\right)^{\frac{1}{\beta - 1}} \right),\nonumber
\end{align}
for all $x,y \in X$ and $t > 0$?
Here, $\beta \geq 2$ is an important value called the \emph{walk dimension}, $c,C > 0$ are some constants and $\{p_t\}_{t > 0}$ is a heat kernel of $(\mathcal{E},\mathcal{F})$; the terminology is introduced in Section \ref{sec:Dir}. When $\beta = 2$, these estimates are called the \emph{Gaussian heat kernel estimates}, which are quite well-understood; we refer to \cite{grigoryan2009heat,saloff2002aspects} and references therein. In this work, we are interested in the strict sub-Gaussian case, $\beta > 2$, which is more involved; see \cite[Introduction]{barlow2004stability} and \cite{kajino2020singularity} for related discussions.

The aforementioned question has been studied extensively both in discrete graph settings;
see the book of Barlow \cite{barlow2017random}, and also on general metric measure spaces  \cite{barlow2006stability, barlow2020stability,barlow2012equivalence,barlow2018stability,grigor2014heat,GrigorCapacity15,grigor2012two,kajino2020singularity,kajino2023conformal}. A result of Barlow, Bass and Kumagai \cite[Theorem 2.16]{barlow2006stability} shows that \ref{eq:HK} is equivalent to the conjunction of the following three conditions.
\begin{enumerate}[label=\textup{(\arabic*)}, leftmargin=*]
    \vspace{3pt}
    \item The measure $\mu$ is \emph{doubling}: There is $D \geq 1$ such that
    \[
        0 < \mu(B(x,2r)) \leq D \mu(B(x,r)) \text{ for all } x \in X \text{ and } r > 0.
    \]
    \item The \emph{Poincar\'e inequality} \ref{intro:PI}: There are $C,\sigma \geq 1$ such that for all
    $x \in X$, $r > 0$ and $f \in \mathcal{F}$,
\begin{equation}\label{intro:PI}
    \tag{$\textup{PI}(\beta)$}
    \int_{B(x,r)} (f - f_{B(x,r)})^2\, d\mu \leq C r^\beta \int_{B(x,\sigma r)} \, d\Gamma\langle f \rangle.
\end{equation}
Here $f_{B(x,r)} := \mu(B(x,r))^{-1} \int_{B(x,r)} f d\mu$ is the integral average and $\Gamma\langle f \rangle$ denotes the energy measure of $f \in \mathcal{F}$; see Section \ref{sec:Dir}.
    \vspace{3pt}
    \item The \emph{cutoff Sobolev inequality} \hyperref[eq:CS(Dir)]{$\textup{CS}_\delta(\beta)$} for some $\delta > 0$; see Definition \ref{def:CS(Dir)}.
\end{enumerate}
This equivalence was first discovered on graphs by Barlow and Bass \cite[Theorem 1.5]{barlow2004stability}, and was later refined by Grigor’yan, Hu, and Lau \cite[Theorem 1.2]{GrigorCapacity15} in the general metric measure space setting. See also \cite[Theorem 4.5]{kajino2023conformal}.

From the practical point of view, however, the characterizations of Barlow--Bass--Kumagai, as well as that of Grigor’yan--Hu--Lau, are rather difficult to use because the literature essentially provides no general sufficient conditions for the cutoff Sobolev inequality except the sub-Gaussian heat kernel estimates.
Some quite mild necessary and sufficient conditions have been established when the ambient space has ``low dimension'' in a certain sense \cite{murugan2023note,barlow2005characterization}. These methods, nevertheless, do not work in the higher dimensional settings.

The technical difficulties surrounding the cutoff Sobolev inequality have been raised in a several works. For instance by Barlow \cite[Remark 3.17]{barlow2013analysis}, by Kumagai \cite[Open problem III]{Kumagai2014}, and by Murugan \cite[Section 6.3]{murugan2024heat}. Moreover, in their \emph{resistance conjecture} (2014), Grigor’yan, Hu, and Lau predicted that, in the aforementioned characterization of sub-Gaussian heat kernel estimates, the cutoff Sobolev inequality can be replaced by a drastically simpler capacity upper bound \cite[Conjecture 4.15]{grigor2014heat}. See \cite[Remark 1.2]{murugan2023note} for the known results to this direction, and Subsection \ref{subsec:concludingremarks} for the precise statement and some further discussion.

\subsection{Main result}
In this work, we give a new characterization of sub-Gaussian heat kernel estimates, in which the cutoff Sobolev inequality is replaced by simpler and more tractable energy upper bounds of cutoff functions. Our condition can be viewed as a variant of the more classical regularity estimate for $\xi \in W^{1,2}(\R^n)$,
\[
    \int_{B(y,r)} \abs{\nabla \xi}^2\, dx \leq C r^{n-2 + \delta}
\]
for all $y \in \R^n$ and $r > 0$, which implies Hölder continuity \cite[Theorem 7.19]{GilbargTrudinger}.
In the Dirichlet form setting, we introduce the following \emph{cutoff energy condition}.

\begin{definition}
    Let $\beta > 2$ and $\delta > 0$. We say that the strongly local regular Dirichlet form $(\mathcal{E},\mathcal{F})$ on $L^2(X,\mu)$ satisfies the \emph{cutoff energy condition} \ref{intro:EC} if there is a constant $C \geq 1$ for which the following holds. For all $x \in X$ and $R > 0$ there is $\xi \in \mathcal{F} \cap C(X)$ satisfying $\xi|_{B(x,R)}=1$, $\xi|_{X \setminus B(x,2R)}=0$, $0 \leq \xi \leq 1$ and
    \begin{equation}\label{intro:EC}
        \tag{$\textup{CE}_{\delta}(\beta)$}
        \int_{B(y,r)} \,d\Gamma\Span{\xi} \leq C\left( \frac{r}{R} \right)^\delta \frac{\mu(B(y,r))}{r^\beta}
    \end{equation}
    for all $y \in X$ and $0 < r \leq 3R$. Here $\Gamma\Span{\xi}$ is the energy measure of $\xi$; see Section \ref{sec:Dir}.
\end{definition}

Unlike the cutoff Sobolev inequality, the cutoff energy condition only involves the cutoff function $\xi$ and its quantitative energy upper bounds. Therefore, it is more intuitive and can be verified directly.
The value $\delta > 0$, as it is shown in Section \ref{sec:Dir}, is related to the Hölder regularity exponent provided by the (a posteriori) parabolic Harnack inequality.

The main result of the work is the following theorem.

\begin{theorem}[Theorem \ref{thm:main}]\label{thm:SUPER}
    Let $(X,d,\mu)$ be a metric measure space where $(X,d)$ is complete, geodesic, locally compact and separable, and $\mu$ is a Radon measure on $(X,d)$ with full topological support.
    Then, a strongly local regular Dirichlet form $(\mathcal{E},\mathcal{F})$ on $L^2(X,\mu)$ satisfies the sub-Gaussian heat kernel estimates \ref{eq:HK} for $\beta > 2$ if and only if $\mu$ is a doubling measure and $(\mathcal{E},\mathcal{F})$ satisfies both the Poincar\'e inequality \ref{intro:PI} and the cutoff energy condition \ref{intro:EC} for some $\delta > 0$.
\end{theorem}

Theorem \ref{thm:SUPER} provides a more intuitive characterization of sub-Gaussian heat kernel estimates than the earlier formulations in the literature that rely on the cutoff Sobolev inequality.
We propose our result as a positive answer to the questions posed in \cite[Remark 3.17]{barlow2013analysis} and \cite[Open Problem III]{Kumagai2014}, both of which ask whether simpler characterizations exist.
Moreover, Theorem \ref{thm:SUPER} also provides a potential approach to study the resistance conjecture.

\subsection{Applications to reflected diffusion}
Beyond providing additional intuition, the cutoff energy condition seems to be more convenient to work with than the cutoff Sobolev inequality. For instance, some technical difficulties concerning the latter were noted by Murugan \cite[Section 6.3]{murugan2024heat} in his study on reflected diffusion.
He showed that the sub-Gaussian heat kernel estimates are inherited from the diffusion process on the ambient space to the reflected diffusion on a uniform domain. Nevertheless, his work left open the question of whether the same result holds for the more general class of inner uniform domains because the proof of the cutoff Sobolev inequality necessarily requires a different approach. We note that, for the Gaussian estimates, this is already known to be true by the work of Gyrya and Saloff-Coste \cite[Theorem 3.10]{SaloffCoste2011NeumannAD}.

To this end, the recent work of the author shows that the cutoff energy condition together with Theorem \ref{thm:SUPER} can be applied to prove that the sub-Gaussian heat kernel estimates are inherited by reflected diffusion on inner uniform domains \cite{AnttilaReflectedDiff}.

\subsection{New examples}
The main inspiration for the cutoff energy condition came from the recent work of the author, Eriksson--Bique and Shimizu \cite{AEBSLaaksoSobolev2025}, which introduced a natural construction of Sobolev space and $p$-energies on Laakso-type fractal spaces; see also the earlier work \cite{anttila2024constructions}. 
Moreover, when $p = 2$, the framework produces strongly local regular Dirichlet forms.
We were, however, unable to verify the sub-Gaussian heat kernel estimates since we did not know how to establish the cutoff Sobolev inequality at the time.
Nevertheless, we were able to obtain the cutoff energy condition, and the first prototype version of Theorem \ref{thm:SUPER} was discovered within the general setting of \cite{AEBSLaaksoSobolev2025}.
In particular, the sub-Gaussian heat kernel estimates are now verified for all examples arising within that framework.

\subsection{Outline of the methods}
The original cutoff Sobolev inequality introduced in \cite{barlow2006stability,barlow2004stability} has been replaced in later works by a priori weaker energy condition called the \emph{simplififed cutoff Sobolev inequality} \cite{andres2015energy}.
However, the results of this work indicate, quite surprisingly, that the original cutoff Sobolev inequality is in fact more transparent and easier to interpret than its simplified counterpart.
We note that, under the doubling property of $\mu$ and the Poincar\'e inequality \ref{intro:PI}, the two variants are equivalent because the simplified version is sufficiently strong to verify \ref{eq:HK} by \cite{GrigorCapacity15}, and the stronger one is implied by \ref{eq:HK} according to \cite{barlow2006stability}.

Our approach to establish Theorem \ref{thm:SUPER} consists of two main ideas. First, we go back to the original cutoff Sobolev inequality. Second, we regard it as a Poincar\'e type inequality and apply Haj{\l}asz--Koskela techniques \cite{SobolevMetPoincare,SobolevMeetsPoincare} and some other common methods from \emph{analysis on metric spaces} literature; see the standard references of analysis on metric spaces \cite{bjorn2011nonlinear,Heinonen,HKST}.
The simplified cutoff Sobolev inequality does not seem to suitable for the second step.

These methods are used in a broader sense than we have indicated so far and several of the results are formulated to apply to a broad class of energies on metric spaces; see for instance Theorem \ref{thm:SobolevPoincare} and Section \ref{Sec:Applications}.
We introduce a general framework of Sobolev spaces and energies, called the \emph{$p$-energy structure}, which is inspired primarily by the notion of \emph{Poincaré inequality pairs} from analysis on metric spaces \cite{SobolevMetPoincare} and by \emph{energy measures} from the Dirichlet form literature \cite{FOT}.
See also \cite{Perez98,KajinoContraction,Perez2018DegeneratePI,Sasaya2025} and therein references for related studies based on similar general frameworks.
We recover some classical results, such as the Sobolev--Poincar\'e inequality, Morrey's inequality and certain results of Chanillo and Wheeden \cite{chanillowheeden85} in weighted Euclidean spaces. We also partially generalize results of Björn and Ka{\l}amajska \cite{BJORNKala} in the classical analysis on metric spaces setting, as our results relax some of their doubling assumptions. See Sections \ref{sec:SPI} and \ref{Sec:Applications} for further details.

\subsection*{Organization of the paper}
In Section \ref{sec:Preli}, we recall some standard terminology and results on metric spaces and measures, and introduce the framework of $p$-energy structures.

Section \ref{sec:SPI} is the most important part of the work, where we study Poincar\'e type inequalities for general Borel measures using the formalism of $p$-energy structures.
In Theorem \ref{thm:SobolevPoincare}, we formulate a condition, that is both necessary and sufficient, for a Borel measure to satisfy a Sobolev--Poincar\'e type inequality. In Remark \ref{rem:dirac}, we show that this result is sharp in certain sense.
Simple applications of Theorem \ref{thm:SobolevPoincare}
are discussed in Section \ref{Sec:Applications}.

In Section \ref{sec:Dir}, we apply the results of Section \ref{sec:SPI} to Dirichlet forms and heat kernel estimates.
We prove the main result of the paper, Theorem \ref{thm:SUPER} and more generally Theorem \ref{thm:main}.
Some other regularity estimates are also studied.

\subsection*{Conventions}
In what follows, given constants $K,L \geq 0$, we frequently write $K \lesssim L$ to indicate the existence of an \emph{inessential parameter} $C \geq 1$ such that $K \leq C \cdot L$. The notations $p$ and $q$ always refer to finite exponents in $[1,\infty)$, and $X$ is always an infinite set.

\section*{Acknowledgments} 
I thank Sylvester Eriksson-Bique, Naotaka Kajino, Mathav Murugan and Ryosuke Shimizu for the discussions on the resistance conjecture and for their helpful comments and suggestions.
This work is supported by the Finnish Ministry of Education and Culture’s Pilot for Doctoral Programmes (Pilot project Mathematics of Sensing, Imaging and Modelling).
The research work was partially conducted during my participation in the Trimster program \emph{Metric Analysis}, organized by Hausdorff Research Institute of Mathematics (HIM) in Bonn, and was financially supported by the Deutsche Forschungsgemeinschaft (DFG, German Research Foundation) under Germany's Excellence Strategy - EXC-2047/1 - 390685813. I express my deepest gratitude towards the institute and local organizers for their hospitality and several of the fellow participants for the engaging discussions related to the present work.
I also thank Pekka Koskela for introducing relevant literature, and Anders Björn and Jana Björn for helpful discussions during their visit to University of Jyväskylä in June 2025. Lastly, I thank Antti Vähäkangas for numerous invaluable discussions and comments on earlier versions of the paper.

\section{Preliminary}\label{sec:Preli}

\subsection{Metric measure spaces}\label{subsec:GeomSetting}
Throughout the paper, unless indicated otherwise, $X = (X,d,\mu)$ always denotes a fixed metric measure space where $d$ is a metric on the set $X$ and $\mu$ is a Radon measure on $(X,d)$.
We shall refer to $\mu$ as the \emph{reference measure}.
We always assume that the metric space $(X,d)$ is separable and locally compact.
We usually refer to the metric measure space $(X,d,\mu)$, or to the metric space $(X,d)$, by the underlying set $X$.

The \emph{open balls} of $X$ are denoted
\[
    B(x,r) := \{ y \in X : d(x,y) < r \} \text{ for } x \in X \text{ and } r > 0.
\]
We denote by $C(X)$ the continuous functions $X \to \R$, and by $C_c(X)$ the subset of compactly supported continuous functions.
We say that a continuous function $\varphi : X \to \R$ is a \emph{cutoff function} for $E \subseteq F$, where $E \subseteq F \subseteq X$ are two non-empty Borel subsets, if $\varphi|_{E} = 1,\ \varphi|_{X \setminus F} = 0$ and $0 \leq \varphi \leq 1$.

For a Borel subset $A \subseteq X$ with $\mu(A) \in(0,\infty)$ and a Borel measurable function $f : X \to \R$, we denote the \emph{integral average}
\[
    f_{A} := \kint_A f \, d\mu := \frac{1}{\mu(A)} \int_A f \,d\mu.
\]
We note that, in some proofs, we simultaneously work with multiple Borel measures. However, since we only take averages with respect to the reference measure $\mu$, we write $f_A$ instead of $f_{A,\mu}$.

We shall impose the following two geometric conditions that are implicitly assumed to hold throughout the paper, unless specified otherwise.

\begin{enumerate}[label=\textup{(\arabic*)}, leftmargin=*]
    \vspace{2pt}
    \item The reference measure $\mu$ is \emph{doubling}, meaning that there is a \emph{doubling constant} $D \geq 1$ of $\mu$ satisfying
\begin{equation}\label{eq:muDoubling}
    0 < \mu(B(x,2r)) \leq D\mu(B(x,r)) < \infty \text{ for all } x \in X \text{ and } r > 0.
\end{equation}
    \vspace{2pt}
    \item The metric space $X$ is \emph{uniformly perfect}, meaning that there is a constant $\lambda \in (0,2]$ such that
\begin{equation}\label{eq:RadDiam}
    \diam(B(x,r)) \geq \lambda r \text{ for all } x \in X \text{ and } r \in (0,\diam(X)).
\end{equation}
Here $\diam(X) := \sup_{x,y \in X} d(x,y)$.
Note that every connected metric space is uniformly perfect.
\end{enumerate}

We recall some properties of doubling measures. See \cite[Chapter 13]{Heinonen} for details.

\begin{lemma}
    There is a constant $N \in \N$ depending only on the doubling constant $D$ of $\mu$ in \eqref{eq:muDoubling} such that the following holds. For all $x \in X$ and $r > 0$ there are points $x_1,\dots,x_N \in X$ satisfying
\begin{equation}\label{eq:MD}
    B(x,2r) \subseteq \bigcup_{i = 1}^N B(x_i,r).
\end{equation}
\end{lemma}
The condition \eqref{eq:MD} is called the \emph{metric doubling property}. It is a standard result that it implies the following \emph{bounded overlapping property}.

\begin{lemma}\label{lemma:Patching}
    Let $\varepsilon > 0$ and $\sigma \geq 1$. Fix any collection $\{x_i\}_{i \in I} \subseteq X$ such that $d(x_i,x_j) \geq \varepsilon$ for all distinct $i,j \in I$, and let
    \[
        G := \bigcup_{i \in I} B(x_i,\sigma \varepsilon).
    \]
    Then $I$ is at most countably infinite and there is a constant $C \geq 1$ depending only on $N$ in \eqref{eq:MD} and $\sigma$ such that
    \begin{equation*}
        \sum_{i \in I} \mathds{1}_{B(x_i,\sigma 2^{-m})} \leq C \mathds{1}_G.
    \end{equation*}
    Here $\mathds{1}_A$ denotes the characteristic function of a Borel set $A \subseteq X$.
\end{lemma}

Lastly, we recall that doubling measures on uniformly perfect metric spaces have the following quantitative dimension bounds.

\begin{lemma}
    Let $\nu$ be any doubling measure on $X$. There are constants $C\geq 1$ and $\alpha_L,\alpha_U \in (0,\infty)$ depending only on the constants in \eqref{eq:muDoubling} and \eqref{eq:RadDiam} such that the following condition holds. Whenever $x,y \in X$ and $0 < r \leq R < \diam(X)$ with $B(y,r) \cap B(x,R) \neq \emptyset$ then
    \begin{equation}\label{eq:RevDoubling}
         C^{-1}\left(\frac{R}{r}\right)^{\alpha_L} \leq \frac{\nu(B(x,R))}{\nu(B(y,r))} \leq C \left(\frac{R}{r}\right)^{\alpha_U}.
    \end{equation}
    
\end{lemma}
If $\nu$ is a doubling measure, we refer to $\alpha_L$ and $\alpha_U$ in \eqref{eq:RevDoubling} as \emph{lower and upper exponents} of $\nu$, respectively.

\subsection{Energy structures}\label{subsec:SobolevSpace}
Next, we introduce the framework of \emph{$p$-energy structures}, which is a general formalism of Sobolev spaces and energies.

First, we need the following definition, which is a slightly modified version of that in \cite[Definition 5.4]{barlow2018stability}.
\begin{definition}\label{def:BSF}
    The term \emph{scale function} refers to any function $\Psi : X \times (0,\infty) \to (0,\infty)$ which satisfies the following doubling type property for some $\beta_U,\beta_L > 0$.
    There is $C \geq 1$ such that for all $x,y \in X$ and $0 < r \leq R < \infty$ such that $B(y,r) \cap B(x,R) \neq \emptyset$,
    \begin{equation}\label{eq:Psi}
        C^{-1}\left(\frac{R}{r}\right)^{\beta_L} \leq \frac{\Psi(x,R)}{\Psi(y,r)} \leq C\left(\frac{R}{r}\right)^{\beta_U}.
    \end{equation}
    We refer to $\beta_L$ and $\beta_U$ as \emph{lower} and \emph{upper exponents} of $\Psi$, respectively.
\end{definition}

\begin{definition}\label{def:SSS}
    Let $p \geq 1$.
    A \emph{$p$-energy structure} on a metric measure space $(X,d,\mu)$ is a triplet $(\mathcal{F}_p,\Gamma_p,\Psi)$ consisting of the following data.
    \begin{enumerate}[label=\textup{(\arabic*)}, leftmargin=*]
    \vspace{3pt}
    \item $\mathcal{F}_p$ is subset of $L^p(X,\mu)$ which we call a \emph{Sobolev space}.
    The members of $\mathcal{F}_p$ are called \emph{Sobolev functions}. 
    \vspace{3pt}
    \item $\Gamma_p$ assigns to every $f \in \mathcal{F}_p$ a non-negative finite Borel measure denoted $\Gamma_p\Span{f}$. We call the measure $\Gamma_p\Span{f}$ the \emph{energy measure} of $f$. To emphasize that these are Dirichlet type $p$-energies in many practical cases, we will frequently use the notation
    \[
         \int_{A} \, d\Gamma_p\Span{f} = \Gamma_p\Span{f}(A).
    \]
    \item $\Psi : X \times (0,\infty) \to (0,\infty)$ is a scale function such that the pair $(\mathcal{F}_p,\Gamma_p)$ satisfies both the Poincar\'e inequality \ref{eq:PI} (Definition \ref{def:PI}) and the upper capacity estimate \ref{eq:UCap} (Definition \ref{def:Ucap}).
    \end{enumerate}
\end{definition}

The precise definitions of the two conditions in Definition \ref{def:SSS}-(3) are postponed to the following subsection.
We first discuss some conventions and examples.

We will often abbreviate and refer to $(\mathcal{F}_p,\Gamma_p,\Psi)$ simply as a $p$-energy structure without explicitly mentioning the metric measure space $(X,d,\mu)$.
We emphasize already here that we do not impose any locality or truncation assumptions, which are quite central concepts in the related literature; see for instance \cite{barlow2018stability,barlow2006stability}. We also do not consider any vector space structure on $\mathcal{F}_p$, even though such a structure is present in all the examples we consider.

\begin{example}\label{ex:Classical}
     Our model example of a $p$-energy structure is the standard Euclidean setting, where $\mathcal{F}_p := W^{1,p}(\R^n)$ and
    \[
        \Gamma_p\Span{f}(A):= \int_A \abs{\nabla f}^p\, dx.
    \]
    On more general metric measure spaces $(X,d,\mu)$, we could consider
    \[
        \Gamma_p\Span{f}(A) := \int_A g_f^p \, d\mu
    \]
    where $g_f$ is some abstract counterpart of $\abs{\nabla f}$.
    A typical choice for $g_f$ in analysis on metric spaces literature is an \emph{upper gradient} of Heinonen--Koskela \cite{HK} or a \emph{Haj{\l}asz gradient} \cite{hajlasz1996sobolev}.
    A common counterpart for $\abs{\nabla f}$ when $f$ is a Lipschitz function is the \emph{pointwise Lipschitz constant}; see \cite[Chapter 8]{HKST},
    \[
        \Lip(f)(x) := \limsup_{r \to 0^+} \sup_{y \in B(x,r)} \frac{\abs{f(x) - f(y)}}{r}.
    \]
    Lastly, we could also consider \emph{two measure} settings, namely we assign
    $\Gamma_p\Span{f}(A) := \int_A g_f^p \, d\nu$ and $\nu$ is a Borel measure of $(X,d)$ different from $\mu$.
    The settings where $\nu$ has a density with respect to $\mu$ is typically called \emph{two-weighted}; see for instance \cite{chanillowheeden85}.
    For further results on two measure settings, see \cite{kinnunenvahakangas2019} and references therein.
\end{example}

\begin{example}
    Our second model example is a Dirichlet form; see Section \ref{sec:Dir} for the precise definitions. Given a strongly local Dirichlet form $(\mathcal{E},\mathcal{F})$ on $L^2(X,\mu)$, where $(X,d,\mu)$ is a sufficiently regular metric measure space, the formula
    \[
    \int_X \varphi \, d\Gamma\Span{f} = \mathcal{E}(f,f\varphi) - \frac{1}{2}\mathcal{E}(f^2,\varphi) \text{ for all } \varphi \in \mathcal{F} \cap C_c(X)
    \]
    uniquely determines Borel measures $\Gamma\Span{f}$ for $f \in \mathcal{F}$.
    Similar approaches to the construction of energy measures for nonlinear counterparts of the Dirichlet form $(\mathcal{E}_p,\mathcal{F}_p)$ with $p \neq 2$ have been studied recently in \cite{KajinoContraction,Sasaya2025}.
\end{example}

\begin{example}\label{example:Non-local}
    The same formalism is applicable to non-local energies as well. For instance, when $\mathcal{F}_p := W^{s,p}(\R^n)$ is the fractional Sobolev space, we can assign $f \in \mathcal{F}_p$ the measure
    \[
        \Gamma_p\Span{f}(A):= \int_{\R^n}\int_A \frac{\abs{f(x) - f(y)}^p}{\abs{x-y}^{sp + n}}\, dxdy.
    \]
   See the book of Maz'ya \cite{Mazya85} for further background.
   Analogous concept can be extended to a doubling metric measure space $(X,d,\mu)$; see for instance \cite{Koskela11}, by assigning
    \[
        \Gamma_p\Span{f}(A) := \int_{X}\int_A \frac{\abs{f(x) - f(y)}^p}{d(x,y)^{sp}\mu(B(y,d(x,y))) }\, d\mu(x)d\mu(y).
    \]
\end{example}

\begin{example}\label{ex:AF}
    The present work is motivated by many examples of natural self-similar energies studied in analysis and diffusions on fractals literature; see \cite{barlow,barlow1989construction,barlow1999brownian,AnalOnFractals,kusuoka1992dirichlet,steinhurst2010diffusions} for Dirichlet forms, \cite{AEBSLaaksoSobolev2025,cao2022p,herman2004p,kigami,murugan2023first,Shimizu} for $p$-energies, and therein references for further examples. See also the recent survey \cite{kajino2025penergyformsfractalsrecent}.
    We note that some of the theories in analysis on fractals are not covered by the analysis on metric spaces literature, and the two fields have mostly developed independently; see the survey \cite{barlow2013analysis} on the topic and also \cite[Introductions]{kigami}.
    For instance, the recent study of the author, Shimizu and Eriksson--Bique \cite{AEBSLaaksoSobolev2025} discovered that, given certain natural Sobolev spaces $\mathcal{F}_p$ for $p > 1$ on the Laakso diamond space, it holds for all $p \neq q$ that $\mathcal{F}_{p} \cap \mathcal{F}_q$ only contains constant function. In other words, Sobolev spaces for distinct exponents are essentially disjoint. Such behavior never happens in the usual settings of analysis on metric spaces simply by the inclusion of Lipschitz functions.
\end{example}

\subsection{Main assumptions}\label{Subsec:PI}
We introduce the precise definitions of the two conditions in Definition \ref{def:SSS}-(3).

\begin{definition}\label{def:PI}
    Fix $p \geq 1$ and let $\mathcal{F}_p$ and $\Gamma_p$ be as in (1)-(2) of Definition \ref{def:SSS}.
    For a given scale function $\Psi$, we say that the pair  $(\mathcal{F}_p,\Gamma_p)$ satisfies the \emph{Poincar\'e inequality} \ref{eq:PI} if there are constants $C,\sigma \geq 1$ satisfying the following condition.
    For all $x \in X$, $r > 0$ and $f \in \mathcal{F}_p$,
    \begin{equation}\label{eq:PI}
        \tag{$\textup{PI}_p(\Psi)$}
        \int_{B(x,r)} \abs{f - f_{B(x,r)}}^p \, d\mu \leq C \Psi(x,r) \int_{B(x,\sigma r)} \, d\Gamma_p\Span{f}.
    \end{equation}
\end{definition}

We present a few familiar examples.

\begin{example}\label{ex:ClassicalPI}
    The condition \ref{eq:PI} can be regarded as a $(p,p)$-Poincar\'e type inquality. 
    In the standard Euclidean Sobolev space $W^{1,p}(\R^n)$ we have
    \[
        \int_{B(x,r)} \abs{f - f_{B(x,r)}}^p \, dx \leq C r^p \int_{B(x,r)} \abs{\nabla f}^p \, dx,
    \]
    which is \ref{eq:PI} for $\Psi(x,r) = r^p$ and $d\Gamma_p\Span{f} = \abs{\nabla f}^p dx$; see \cite{saloff2002aspects}.
    For the fractional Sobolev space $W^{s,p}(\R^n)$, it is a direct computation to show that
    \[
        \int_{B(x,r)} \abs{f - f_{B(x,r)}}^p \leq Cr^{sp} \int_{B(x,r)}\int_{B(x,r)} \frac{\abs{f(x) - f(y)}^p}{\abs{x - y}^{n + sp}}\, dxdy,
    \]
    which implies \ref{eq:PI} for $\Psi(x,r) = r^{sp}$ and $\Gamma_p\Span{f}$ would be as in Example \ref{example:Non-local}.
\end{example}

\begin{example}\label{ex:TwoMeasure}
    Two measure $(p,p)$-Poincar\'e inequalities in metric spaces are typically of the form
    \[
        \kint_{B(x,r)} \abs{f - f_{B(x,r)}}^p\, d\mu \leq Cr^p \kint_{B(x,r)} g_f^p \, d\nu
    \]
    where $\mu$ and $\nu$ are doubling measures on $X$ and $g_f$ is some appropriate counterpart of $\abs{\nabla f}$; see for instance \cite{kinnunenvahakangas2019,chanillowheeden85}.
    This is the Poincar\'e inequality \ref{eq:PI} for
    \[
        \Psi(x,r) := r^p \frac{\mu(B(x,r))}{\nu(B(x,r))}.
    \]
    In particular, $\Psi(x,r)$ may also depend on $x \in X$ and not only $r$.
\end{example}

\begin{definition}\label{def:Ucap}
    Fix $p \geq 1$ let $\mathcal{F}_p$ and $\Gamma_p$ be as in (1)-(2) of Definition \ref{def:SSS}.
    For a given scale function $\Psi$, we say that the pair $(\mathcal{F}_p,\Gamma_p)$ satisfies the \emph{upper capacity estimate} \ref{eq:UCap} if there are constants $C,\kappa > 1$ satisfying the following condition.
    For every $x \in X$ and $r \in (0,\diam(X))$ there is a cutoff function $\varphi \in \mathcal{F}_p$ for $B(x,r) \subseteq B(x,\kappa r)$ satisfying 
    \begin{equation}\label{eq:UCap}
        \tag{$\textup{Cap}_{p,\leq}(\Psi)$}
        \int_X \, d\Gamma_p\Span{\varphi} \leq C \frac{\mu(B(x,r))}{\Psi(x,r)}.
    \end{equation}
\end{definition}

\begin{example}\label{ex:EuclideanUcap}
In the standard Euclidean setting $\mathcal{F}_p : = W^{1,p}(\R^n)$ and $\Psi(x,r) := r^p$, the upper capacity estimate \ref{eq:UCap} is obtained via the cutoff functions
\[
    \varphi(y) := \left(\frac{2r - \abs{x - y}}{r} \lor 0\right) \land 1.
\]
This can be seen by noting that $\varphi$ satisfies $\abs{\nabla \varphi} \leq r^{-1} \mathds{1}_{B(x,2r)}$ where $\mathds{1}_{B(x,2r)}$ is the characteristic function.
See \cite{BBCapacity23,BBLJUHA} and references therein for capacity estimates in other settings.
\end{example}

We briefly elaborate the energy estimate in \ref{eq:UCap}, and see the proof of Proposition \ref{prop:Converse} for further details.
Assume that $\varphi$ is a cutoff function for $B(x,r) \subseteq B(x,\kappa r)$ where $\kappa > 1$.
If the Poincar\'e inequality \ref{eq:PI} were to hold, we can use the doubling property of $\mu$ to show that for sufficiently small open balls $B(x,r) \subseteq X$,
\[
    \mu(B(x,r)) \lesssim \int_{B(x,\tau r)} \abs{\varphi - \varphi_{B(x,\tau r)}}^p \, d\mu \lesssim \Psi(x,r) \Gamma_p\Span{\varphi}(X)
\]
where $\tau \geq \kappa$ is some fixed constant depending only on the doubling constant of $\mu$ and $\kappa$.
From this, we obtain the energy lower bound
\[
    \Gamma_p\Span{\varphi}(X) \gtrsim  \frac{\mu(B(x,r))}{\Psi(x,r)}.
\]
The upper capacity estimate \ref{eq:UCap} thus asserts that the lower bound of the energy $\Gamma_p\Span{\varphi}(X)$ provided by the Poincar\'e inequality is sharp for some cutoff functions.

\subsection{Auxiliary lemmas}
We present two lemmas which are helpful in the next section.

\begin{lemma}\label{lemma3}
    Let $p \geq 1$ and $(\mathcal{F}_p,\Gamma_p,\Psi)$ be a $p$-energy structure.
    Let $\tau \geq 1$, $x,y \in X$ and $r,s > 0$ such that
    $B(x,r) \cap B(y,s) \neq \emptyset$ and $\tau^{-1}s \leq r \leq \tau s$.
    There are constants $C = C,\sigma \geq 1$ depending only on $\tau$, $p$, the constants in \ref{eq:PI}, the constants associated to $\Psi$ in \eqref{eq:Psi}, the constants in \eqref{eq:muDoubling} and \eqref{eq:RadDiam} such that for all $f \in \mathcal{F}_p$,
    \begin{equation}\label{eq:DifferenceAverages}
        \abs{f_{B(x,r)} - f_{B(y,s)}} \leq C \left(\frac{\Psi(x,r)}{\mu(B(x,r))}
        \int_{B(x,\sigma r)} \, d\Gamma_p\Span{f} \right)^{\frac{1}{p}}.
    \end{equation}
\end{lemma}

\begin{proof}
    Choose a point $z \in B(x,r) \cap B(y,s)$ and let $R := 2\tau r$.
    By noting that $B(x,r),B(y,s) \subseteq B(z,R)$ and that $R,r,s$ are all comparable,
    \begin{align*}
        & \quad \, \left\lvert f_{B(x,r)} - f_{B(y,s)} \right\rvert\\
        \leq & \quad \left\lvert f_{B(x,r)} - f_{B(z,R)} \right\rvert + \left\lvert f_{B(y,s)} - f_{B(z,R)} \right\rvert \\
        \leq & \quad \kint_{B(x,r)} \abs{f -f_{B(z,R)}} \,d\mu + \kint_{B(y,s)} \abs{f -f_{B(z,R)}} \, d\mu 
        \\
        \lesssim & \quad \kint_{B(z,R)} \abs{f - f_{B(z,R)}}\, d\mu && (\text{Doubling property \eqref{eq:muDoubling}}) \\
        \leq & \quad \left(\kint_{B(z,R)} \abs{f - f_{B(z,R)}}^p \, d\mu\right)^{\frac{1}{p}} && (\text{Hölder's ineq.})\\
        \lesssim & \quad \left(\frac{\Psi(z,R)}{\mu(B(z,R))}
        \int_{B(z,\sigma_0 R)} \, d\Gamma_p\Span{f} \right)^{\frac{1}{p}} && \text{(Poincar\'e ineq. \ref{eq:PI})}\\
        \lesssim & \quad \left(\frac{\Psi(x,r)}{\mu(B(x,r))}
        \int_{B(x,\sigma r)} \, d\Gamma_p\Span{f} \right)^{\frac{1}{p}}. && (\eqref{eq:muDoubling} \text{ and } \eqref{eq:Psi})
    \end{align*}
    We chose a suitable $\sigma$, which depends on $\tau$ and $\sigma_0$, in the last row of the previous display.
\end{proof}

We also introduce a general notion of Lebesgue differentiation.
The proof is based on a standard maximal function technique; see \cite[Proof of Theorem 5.1]{bjorn2011nonlinear} and \cite[Lemma 4.5]{BJORNKala} for similar arguments.

\begin{lemma}\label{lemma:Lebesguept}
    Let $1 \leq p \leq q$, $\Theta$ be a scale function and $(\mathcal{F}_p,\Gamma_p,\Psi)$ be a $p$-energy structure.
    Let $\nu$ be a Borel measure on $X$ for which there are constants $K \geq 1$ and $R > 0$ such that
        \begin{equation}\label{eq:Lebesgue1}
            \nu(B(x,r))^{\frac{1}{q}} \leq K \Theta(x,r) \left( \frac{\mu(B(x,r))}{\Psi(x,r)} \right)^{\frac{1}{p}}
        \end{equation}
    for all $x \in X$ and $0 < r \leq R$.
    Then for every $f \in \mathcal{F}_p$ the following holds for $\nu$-almost every $x \in X$. Given any pair of sequences $y_i\in X$ and $r_i > 0$ such that $x \in B(y_i,r_i)$ for all $i \in \N$ and $r_i \to 0^+$ as $i \to \infty$, the limit
    \begin{equation}\label{eq:Lebesgue2}
        \lim_{i \to \infty} \kint_{B(y_i,r_i)} f\, d\mu
    \end{equation}
    exists and is independent of the choices of the sequences.
\end{lemma}

\begin{proof}
    Fix $f \in \mathcal{F}_p$.
    Consider the following variant of the maximal function,
    \[
        M(x) := \sup_{x \in B(y,r)} \left( \Theta(y,r)^{p}\frac{\mu(B(y,r))}{\Psi(y,r)} \right)^{-1} \int_{B(y,r)}\, d\Gamma_p\Span{f},
    \]
    where the supremum is taken over all $y \in Y$ and $r \in (0,R/5)$ such that $x \in B(y,r)$.

    We first show that $M(x) < \infty$ for $\nu$-almost every $x \in X$. The idea of the proof is very standard; see \cite[Lemma 3.12]{bjorn2011nonlinear}.
    Fix $s > 0$. Because $M$ is a non-centered maximal type function, it is lower semicontinuous, and in particular the level set
    \[
        E_s := \{ x \in X : M(x) > s \} \subseteq X
    \]
    is open. Then for every $x \in E_s$, there is an open ball $x \in B(z_x,r_x) \subseteq E_s$ where $r_x \in (0,R/5)$ such that
    \[
        \Theta(z_x,r_x)^{p} \frac{\mu(B(z_x,r_x))}{\Psi(z_x,r_x)} < \frac{1}{s}\int_{B(z_x,r_x)}\, d\Gamma_p\Span{f}.
    \]
    By $5r$-covering lemma; see \cite[Lemma 1.7]{bjorn2011nonlinear}, there is a countable family $I \subseteq E_s$ such that $\{ B(z_x,r_x) \}_{x \in I}$ is disjoint and $\{ B(z_x,5r_x) \}_{x \in I}$ covers $E_s$.
    
    Note that $q \geq p$ implies $(a + b)^{p/q} \leq a^{p/q} + b^{p/q}$ for all $a,b \geq 0$. We use this, the assumption of the lemma \eqref{eq:Lebesgue1}, the doubling property of $\mu$ \eqref{eq:muDoubling} and the doubling properties of $\Psi,\Theta$ \eqref{eq:Psi},
    \begin{align*}
        \nu(E_s)^{\frac{p}{q}} & \leq \sum_{x \in I} \nu(B(z_x,5r_x))^{\frac{p}{q}} \lesssim \sum_{x \in I}\Theta(B(z_x,5r_x))^p \frac{\mu(B(z_x,5r_x))}{\Psi(z_x,5r_x)}\\
        & \lesssim \sum_{x \in I}\Theta(z_x,r_x)^p \frac{\mu(B(z_x,r_x))}{\Psi(z_x,r_x)}.
    \end{align*}
    The combination of the inequalities in the previous two displays along with the facts that $\Gamma_p\Span{f}$ is a Borel measure and $\{ B(z_x,r_x) \}_{x \in I}$ is a countable disjoint collection of Borel sets,
    \begin{align*}
        \nu(E_s)^{\frac{p}{q}} & \lesssim \frac{1}{s} \int_{X} \, d\Gamma_p\Span{f}.
    \end{align*}
    Since $\Gamma_p\Span{f}$ is assumed to be a finite Borel measure, we see that $M(x) < \infty$ for $\nu$-almost every $x \in X$ by letting $s \to \infty$.

    By the previous part of the proof, it is sufficient to prove the conclusion of the current lemma for all points $x \in X$ where $M(x) < \infty$. We do this by establishing the following estimate for every $x \in X$. If $y \in X$ and $0 < s < r \leq R/(5\sigma)$, where $R$ is as in the claim and $\sigma$ as in \ref{eq:PI}, then whenever $x \in B(y,s)$,
    \[
        \abs{f_{B(x,r)} - f_{B(y,s)}} \lesssim \left(\frac{r}{R} \right)^{\delta}\Theta(x,R)M(x)^{\frac{1}{p}},
    \]
    where $\delta > 0$ is a lower exponent of $\Theta$ in \eqref{eq:Psi}. The desired convergence would follow by letting $r,s \to 0^+$.

    First, assume $0 < r/2 \leq s \leq r \leq R/(5\sigma)$.
    It follows from Lemma \ref{lemma3} and the present assumptions,
    \begin{align*}
        \abs{f_{B(x,r)} - f_{B(y,s)}}
        & \lesssim \Theta(x,r)\left( \left( \Theta(x,\sigma r)^{p}\frac{\mu(B(x,\sigma r))}{\Psi(x,\sigma r)} \right)^{-1} \int_{B(x,\sigma r)} \, d\Gamma_p\Span{f}  \right)^{\frac{1}{p}}\\
        & \leq \Theta(x,r) M(x)^{\frac{1}{p}} \lesssim \left(\frac{r}{R}\right)^{\delta} \Theta(x,R) M(x)^{\frac{1}{p}}.
    \end{align*}
    Next, assume $0 < r/2^{k+1} \leq s \leq r/2^k < \infty$ for $k \in \N$. By applying the previous inequality $k+2$ times,
    \begin{align*}
        \abs{f_{B(x,r)} - f_{B(y,s)}} & \leq \sum_{i = 0}^k \abs{f_{B(x,2^{-i}r)} - f_{B(x,2^{-(i+1)}r}} + \abs{f_{B(x,2^{-(k+1)}r)} - f_{B(y,s)}}\\
        & \lesssim \left(\frac{r}{R}\right)^{\delta} \Theta(x,R)M(x)^{\frac{1}{p}} \sum_{i = 0}^{k+1} 2^{-i\delta} \lesssim \left(\frac{r}{R}\right)^{\delta} \Theta(x,R)M(x)^{\frac{1}{p}}.
    \end{align*}
    This completes the proof.
\end{proof}

Before moving to the next section, we need one additional tool to analyze general Borel measures, and that is to choose suitable point-wise defined $\mu$-representatives of Sobolev functions. We need them to integrate Sobolev functions against measures which are not absolutely continuous with respect to $\mu$.
Nevertheless, any choice performs equally well as long as the they are Borel measurable, $\tilde{f} = f$ whenever $f$ is continuous, and the telescoping argument in the proof of Proposition \ref{prop:LocalUD} is valid.

For a given $f \in \mathcal{F}_p$, its \emph{precise representative} $\tilde{f} : X \to [-\infty,\infty]$ is the point-wise defined Borel function
\begin{equation}\label{eq:Exact}
    \tilde{f}(x) := \limsup_{i \to \infty} \sup_{x \in B(y,2^{-i})}\kint_{B(y,2^{-i})} f\, d\mu,
\end{equation}
where the supremum inside the limit superior is taken over all $y \in X$ such that $x \in B(y,2^{-i})$.
We note that $\tilde{f}$ is a point-wise limit of lower-semicontinuous functions, hence Borel measurable. By the Lebesgue differentiation theorem in metric spaces, which holds for instance when the ambient space is locally compact and the measure is doubling, $\tilde{f}(x) = f(x)$ for $\mu$-almost every $x \in X$; we refer to \cite[Chapter 2]{Heinonen} and \cite[Chapter 14]{SobolevMetPoincare} for further details.

\section{Poincar\'e inequalities for general measures}\label{sec:SPI}
In this section, we study general Poincar\'e type inequalities. The proofs here rely heavily on techniques from the analysis on metric spaces literature.
Throughout the section, we use the same notations and conventions as in the previous section. Namely, $X = (X,d,\mu)$ is a metric measure where $(X,d)$ is a uniformly perfect metric space and $\mu$ is a doubling measure on $(X,d)$.

We now state the main theorem of the section.

\begin{theorem}\label{thm:SobolevPoincare}
    Let $1 \leq p \leq q$, $(\mathcal{F}_p,\Gamma_p,\Psi)$ be a $p$-energy structure, $\Theta$ be a scale function and $\nu$ be a Borel measure of $X$.
    Then the following two conditions are equivalent.
    \begin{enumerate}[label=\textup{(T\arabic*)}, leftmargin=*]

        \vspace{2pt}
        \item \label{T1}
        There is a constant $K \geq 1$ such that for all $x \in X$ and $r \in (0,\diam(X))$,
        \begin{equation*}
            \nu(B(x,r))^{\frac{1}{q}} \leq K \Theta(x,r) \left( \frac{\mu(B(x,r))}{\Psi(x,r)} \right)^{\frac{1}{p}}
        \end{equation*}
        \vspace{2pt}
        \item \label{T2}
        There are constants $C,\sigma \geq 1$ such that for all $x \in X$, $r \in (0,\diam(X))$ and $f \in \mathcal{F}_p$,
        \begin{equation*}
    \left(\int_{B(x,r)} \abs{\tilde{f} - f_{B(x,r)}}^{q} \, d\nu \right)^{\frac{1}{q}} \leq C \Theta(x,r) \left(\int_{B(x,\sigma r)} d\Gamma_p\Span{f} \right)^{\frac{1}{p}}
        \end{equation*}
    \end{enumerate}
    Here $\tilde{f}$ is the precise representative given in \eqref{eq:Exact}.
    This equivalence is quantitative in the sense that the constants appearing in \ref{T2} \textup{(}resp. \ref{T1}\textup{)} depend only on the constants in \ref{T1} \textup{(}resp. \ref{T2}\textup{)}, the constants appearing in \ref{eq:PI},\ref{eq:UCap}, \eqref{eq:muDoubling}, \eqref{eq:RadDiam}, the constants associated with $\Psi$ and $\Theta$ in \eqref{eq:Psi}, $p$ and $q$.
\end{theorem}

\begin{remark}
    We have prepared some examples to Theorem \ref{thm:SobolevPoincare} in Section \ref{Sec:Applications}. It might be helpful for the reader to consult these before proceeding.
\end{remark}

\begin{remark}
    The averaging constant in the left hand side of \ref{T2}, namely $f_{B(x,r)} = \mu(B(x,r))^{-1}\int_{B(x,r)} f d\mu$, is taken with respect to $\mu$ instead of $\nu$. This is an essential detail, for instance, in the proof of Proposition \ref{prop:CSandCE}.
\end{remark}

The inequality \ref{T2} is a Sobolev--Poincar\'e type inequality in a quite general form.
In Remarks \ref{rem:delta=0} and \ref{rem:dirac}, we show that Theorem \ref{thm:SobolevPoincare} is sharp in certain sense.
These kinds of inequalities are very common in many different settings of analysis; see \cite{saloff2002aspects} for the classical inequality, \cite{chanillowheeden85,Perez98,Perez2018DegeneratePI} for works in a weighted Euclidean settings and \cite[Introduction]{SobolevMetPoincare} and therein references for studies in more general contexts.
See also \cite{bakry1995sobolev} for related results on Dirichlet forms.

The condition \ref{T1} has also appeared on several occasions in the literature.
For instance, it arises in the Adams inequality; see \cite{adams1973trace,makalainen2009adams} and \cite[Section 1.4.1]{Mazya85}, and in trace embeddings into Besov spaces \cite{jonsson1980trace,jonsson1984function}.
It can sometimes be understood as an isoperimetric/isocapacitary type inequality for open balls \cite{KinnunenKorte08,Mazya85}, and it is closely related to the balance condition of Chanillo--Wheeden \cite{chanillowheeden85,kinnunenvahakangas2019}.

A recent work by J. Björn and Ka{\l}amajska \cite[Theorem 4.1]{BJORNKala} studied conditions similar to \ref{T1} and their connection to Sobolev--Poincar\'e inequalities which in some cases can be identified as \ref{T2}. However, their framework contains certain doubling assumptions on the measure $\nu$ whereas the present work relaxes this requirement; see \cite[Assumption (\textbf{D})]{BJORNKala}.
The weaker doubling assumptions are essential for our main result.
We also obtain a simple proof for the Morrey's inequality in Proposition \ref{prop:Morrey} by applying Theorem \ref{thm:SobolevPoincare} to the Dirac delta measures, which do not satisfy the doubling conditions in \cite{BJORNKala}.
On the other hand, under these doubling assumption, \cite{BJORNKala} obtains certain endpoint estimates which our theorem does not reach; see Section \ref{Sec:Applications} for further discussion.

Lastly, we mention that a quite similar condition was studied in the Dirichlet form setting in Barlow--Murugan \cite[Definition 4.1]{barlow2018stability} and in Barlow--Chen--Murugan \cite[Definition 6.2]{barlow2020stability}, although for a rather different purpose.

\subsection{Proof of Theorem \ref{thm:SobolevPoincare}; Part \textup{(I)}}
\label{subsec:MainProof}

The first implication 
\ref{T1} $\Rightarrow$ \ref{T2} is proven via a covering argument. The method is primarily inspired by Haj{\l}asz--Koskela \cite{SobolevMeetsPoincare,SobolevMetPoincare} but we work with relaxed doubling assumptions.

We first skim over the required tools.
For every $n \in \Z$ we fix a \emph{a maximal $2^{-n}$ separated subset} $V_n \subseteq X$, which means
\[
    d(v,w) \geq 2^{-n} \text{ for every pair of distinct } v,w\in V_n
\]
and
\[
   X = \bigcup_{v \in V_n} B(v,2^{-n}).
\]
Note that the sets $V_n$ are always at most countably infinite since $X$ is assumed to be separable.

We shall prove the first implication by deriving the following quantitative local Poincar\'e inequality. We do not need \ref{eq:UCap} here.

\begin{proposition}\label{prop:LocalUD}
    Let $1 \leq p \leq q$, $(\mathcal{F}_p,\Gamma_p,\Psi)$ be a $p$-energy structure and $\nu$ be a Borel measure of $X$.
    Assume that the pair $x_0 \in X$ and $R_0 \in (0,\diam(X))$ satisfies the following condition. There are constants $L_0\geq 1$ and $\delta > 0$ such that the inequality
    \begin{equation}\label{eq:LocalUD}
        \nu(B(x,r))^{\frac{1}{q}} \leq L_0 \left(\frac{r}{R_0}\right)^{\delta}\left(\frac{\mu(B(x,r))}{\Psi(x,r)}\right)^{\frac{1}{p}}
    \end{equation}
    holds for all $x \in X$ and $r > 0$ satisfying
    \[
        B(x,r) \cap B(x_0,R_0) \neq \emptyset \text{ and }0 < r \leq R_0.
    \]
    Then there are constants $C,\sigma \geq 1$ such that for all $f \in \mathcal{F}_p$,
    \[
        \left(\int_{B(x_0,R_0)} \abs{\tilde{f} - f_{B(x_0,R_0)}}^{q} \, d\nu \right)^{\frac{1}{q}} \leq C L_0 \left(\int_{B(x_0,\sigma R_0)} d\Gamma_p\Span{f} \right)^{\frac{1}{p}}.
    \]
    This result is quantitative in the sense that $C$ and $\sigma$ depend only on the constants in \ref{eq:PI}, \eqref{eq:muDoubling}, \eqref{eq:RadDiam}, the constant associated to $\Psi$ in \eqref{eq:Psi}, $\delta, \, p$ and $q$.
\end{proposition}

\begin{remark}\label{rem:delta=0}
    The constant $C$ in Proposition \ref{prop:LocalUD} we obtain in the proof blows up to $\infty$ as $\delta \to 0^+$ because of the estimate \eqref{eq:MT2}. Moreover, Proposition \ref{prop:LocalUD} in general is false if $\delta = 0$; see Remark \ref{rem:dirac}. We also note that the proof for the converse direction in Proposition \ref{prop:Converse} works even with $\delta= 0$. Lastly, we note that Proposition \ref{prop:LocalUD} sometimes holds even for $\delta = 0$; see \cite[Theorem 4.1]{BJORNKala}.
\end{remark}

\begin{remark}
    A standard approach for obtaining Sobolev type inequalities in the literature is to first derive analogous weak type estimates and then convert them to the desired strong type via the \emph{truncation technique} of Maz'ya \cite{Mazya85}; we refer to \cite{SobolevMetPoincare,bakry1995sobolev} for details. In the following proof, we obtain the strong inequality directly without passing through the weak type estimates.
\end{remark}

\begin{proof}[Proof of Proposition \ref{prop:LocalUD}]    
    We shall do some preparations first.
    Let $m_0 \in \Z$ be the smallest integer such that $2^{-m_0} \leq R_0 \leq 2^{-m_0 + 1}$.
    For every $i \in \N \cup \{0\}$ we define the subset $W_{m_0+i} \subseteq V_{m_0+i}$ that consists of the points $v \in V_{m_0+i}$ such that $B(v,2^{-(m_0+i)}) \cap B(x_0,R_0) \neq \emptyset$.
    Lastly, every $x \in B(x_0,R_0)$ is assigned a sequence of open balls $\{B_{i,x}\}_{i = 1}^{\infty}$ given by 
    \[
        B_{i,x} := 
        \begin{cases}
            B(x_0,R_0) & \text{ if } i = 0\\
            B(x,2^{-m_0-i}) & \text{ if } i > 0.
        \end{cases}
    \]

    We now begin the actual proof.
    By Lemma \ref{lemma:Lebesguept} and the definition of $\tilde{f}$ in \eqref{eq:Exact},
    \[
        \lim_{i \to \infty} \abs{\tilde{f}(x) - f_{B_{i,x}}} = 0 \text{ for $\nu$-almost every } x \in X.
    \]
    Thus, for the rest of the proof, we only consider points $x \in X$ for which the previous equality holds.
    A telescoping argument now yields
    \begin{equation}\label{eq:MT1}
        \abs{\tilde{f}(x) - f_{B(x_0,R_0)}} \leq  \sum_{i = 0}^{\infty} \abs{f_{B_{i,x}} - f_{B_{i+1,x}}}.
    \end{equation}
    
    Next, we set $\varepsilon := q\delta/2$. 
    If $q > 1$, we proceed by applying a weighted Hölder's inequality,
    \begin{align}
        \label{eq:MT2}
        & \, \quad \abs{\tilde{f}(x) - f_{B(x_0,R_0)}}^{q} \\
        \leq & \quad \left(\sum_{i = 0}^{\infty} \abs{f_{B_{i,x}} - f_{B_{i+1,x}}}2^{\frac{i \varepsilon}{q}} 2^{\frac{-i \varepsilon}{q}} \right)^{q} && (\text{Ineq. \eqref{eq:MT1}}) \nonumber
        \\
        \leq & \quad \left(\sum_{i = 0}^{\infty} \left(2^{\frac{-\varepsilon}{q-1}} \right)^{i}\right)^{q-1}\sum_{i = 0}^{\infty}\abs{f_{B_{i,x}} - f_{B_{i+1,x}}}^{q} 2^{i \varepsilon} && (\text{Hölder's ineq.}) \nonumber
        \\
        \lesssim & \quad \sum_{i = 0}^{\infty}\abs{f_{B_{i,x}} - f_{B_{i+1,x}}}^{q} 2^{i \varepsilon}. && (\varepsilon > 0)
        \nonumber
    \end{align}
    If $q = 1$, then in \eqref{eq:MT2} the first row is trivially less than the last row due to \eqref{eq:MT1}.
    
    The next step is to estimate the differences of the averages using the Poincar\'e inequality.
    Let $w \in W_{m_0 + i}$ such that $x \in B(w,2^{-(m_0 + i)})$. Note that the balls $B_{i,x},B_{i+1,x}$ and $B(w,2^{-(m_0 + i)})$ all have comparable radii and all contain $x$. By Lemma \ref{lemma3}, there is $\sigma_0 \geq 1$ independent of $x, \, w$ and $i$ such that
    \begin{align*}
        \abs{f_{B_{i,x}} - f_{B_{i+1,x}}}^{q} & \lesssim \abs{f_{B(w,2^{-(m_0+i)})} - f_{B_{i,x}}}^q + \abs{f_{B(w,2^{-(m_0 + i)})} - f_{B_{i+1,x}}}^q\\
        & \lesssim \left(\frac{\Psi(w,2^{-(m_0+i)}) }{\mu(B(w,2^{-(m_0+i)}))}
        \int_{B(w,\sigma_0 2^{-(m_0+i)}))} \, d\Gamma_p\Span{f} \right)^{\frac{q}{p}}. \nonumber
    \end{align*}
    We shall denote $B_{i,w} := B(w,2^{-(m_0+i)}))$ and $\sigma_0 B_{i,w} := B(w,\sigma_0 2^{-(m_0+i)}))$ to shorten the notation.
    By combining the previous two displays, we get for $\nu$-almost every $x \in X$,
    \begin{align*}
        & \abs{\tilde{f}(x) - f_{B(x_0,R_0)}}^q  \lesssim \sum_{i = 0}^\infty  \sum_{w \in W_{m_0+i}} \left(\frac{\Psi(w,2^{-(m_0+i)}) }{\mu(B_{i,w})}
        \int_{\sigma_0 B_{i,w}} \, d\Gamma_p\Span{f} \right)^{\frac{q}{p}} 2^{i\varepsilon}\mathds{1}_{B_{i,w}}(x).
    \end{align*}
    
    Now, we integrate over the set $B(x_0,R_0)$,
    \begin{align*}
        & \, \quad
        \int_{B(x_0,R_0)}\abs{\tilde{f} - f_{B(x_0,R_0)}}^{q} \,d\nu \\
        \lesssim & \quad \sum_{i = 0}^{\infty}\sum_{w \in W_{m_0 + i}} \int_{B(x_0,R_0)} \left(\frac{\Psi(w,2^{-(m_0+i)})}{\mu(B_{i,w})}\right)^{\frac{q}{p}} \left(
        \int_{B_{i,w}} \, d\Gamma_p\Span{f} \right)^{\frac{q}{p}} 2^{i\varepsilon}\mathds{1}_{B_{i,w}} \, d\nu \\
        \leq & \quad \sum_{i = 0}^{\infty} \sum_{w \in W_{m_0 + i}} \nu(B_{i,w})\left(\frac{\Psi(w,2^{-(m_0+i)})}{\mu(B_{i,w})}\right)^{\frac{q}{p}} \left(
        \int_{\sigma_0 B_{i,w}} \, d\Gamma_p\Span{f} \right)^{\frac{q}{p}} 2^{i\varepsilon},
    \end{align*}
    Since each $B_{i,w}$ has a radius $2^{-(m_0+i)}$ and $2^{-m_0} \leq R_0$, we can apply \eqref{eq:LocalUD} to the last row in the previous display. We obtain
\begin{equation}\label{eq:integrate}
    \int_{B(x_0,R_0)}\abs{\tilde{f} - f_{B(x_0,R_0)}}^{q} \,d\nu
        \lesssim L_0^q\sum_{i = 0}^{\infty} \sum_{w \in W_{m_0 + i}} 2^{-i(q\delta - \varepsilon)} \left(
        \int_{\sigma_0 B_{i,w}} \, d\Gamma_p\Span{f} \right)^{\frac{q}{p}}.
    \end{equation}
    
Note that $\sigma_0 B_{i,w} \subseteq B(x_0,3\sigma_0 R_0)$. Also recall that the balls $B_{i,w}$ for $w \in W_{m_0 + i}$ are centered at points that are $2^{-(m_0+i)}$-separated.
By combining these facts with Lemma \ref{lemma:Patching} and the inequality $a^{q/p} + b^{q/p} \leq (a+b)^{q/p}$ where $a,b \geq 0$,
    \begin{align}\label{eq:UsePatching}
         \sum_{w \in W_{m_0 + i}}  \left(
        \int_{\sigma_0 B_{i,w}} \, d\Gamma_p\Span{f} \right)^{\frac{q}{p}}  & \leq \left(\sum_{w \in W_{m_0 + i}}\int_{\sigma_0 B_{i,w}} \, d\Gamma_p\Span{f} \right)^{\frac{q}{p}}\\
        & \lesssim \left(\int_{B(x_0,3\sigma_0 R_0 )} \, d\Gamma_p\Span{f}\right)^{\frac{q}{p}}\nonumber
    \end{align}
We note that $q \geq p$ is needed here.
    
    We are finally ready to finish the proof. Since we had chosen $\varepsilon = q\delta/2$, the combination of the inequalities \eqref{eq:integrate} and \eqref{eq:UsePatching} yields
    \begin{align*}
        \int_{B(x_0,R_0)}\abs{\tilde{f} - f_{B(x_0,R_0)}}^{q} \,d\nu & \lesssim L_0^q \sum_{i = 0}^{\infty} \sum_{w \in W_{m_0 + i}} 2^{-i( q\delta - \varepsilon)} \left(\int_{\sigma_0 B_{i,w}} \, d\Gamma_p\Span{f} \right)^{\frac{q}{p}} \\
        & = L_0^q \sum_{i = 0}^{\infty} 2^{-iq\frac{\delta}{2}} \sum_{w \in W_{m_0 + i}}\left(\int_{\sigma_0 B_{i,w}} \, d\Gamma_p\Span{f} \right)^{\frac{q}{p}} \\
        & \lesssim L_0^q \left(\int_{B(x_0,3\sigma_0 R_0)} \, d\Gamma_p\Span{f} \right)^{\frac{q}{p}}
    \end{align*}
    The proof is completed after raising the inequalities in the previous display to the power $1/q$.
\end{proof}

\begin{remark}\label{rem:SobolevMeetsPoincare}
The proof of Proposition \ref{prop:LocalUD} has a slightly hidden similarity with the technique employed in the note of Haj{\l}asz and Koskela \cite{SobolevMeetsPoincare}, even though on the level of details there are quite major differences.
Their proof is based on a very simple observation that we describe here.  Let $f \in L^p(X,\mu)$ and fix $x_0 \in X$ and $R_0>0$. Fix a constant $s > 0$ and choose any point $z \in \{ y \in B(x_0,R_0) : \abs{f(y) - f_{B(x_0,R_0)}} > s \}$ be a Lebesgue point of $f$.
By the telescoping argument in \eqref{eq:MT1},
\[
    s < \abs{f(z) - f_{B(x_0,R_0)}} \leq \sum_{i = 0}^{\infty} \abs{f_{B_{i}} - f_{B_{i+1}}}
\]
where $B_i := B(z,2^{-i}r)$. Now fix a parameter $\varepsilon > 0$.
By modifying the left hand side of the previous inequality, we see that
\[
    \sum_{i = 0}^\infty s 2^{-i\varepsilon} < \sum_{i = 0}^{\infty} C(\varepsilon) \abs{f_{B_{i}} - f_{B_{i+1}}}.
\]
Since we have an inequality between two infinite series, there in particular has to be at least one index $i$ such that we have an inequality between the $i$-th terms. By rewriting, we arrive at
\[
s < C(\varepsilon) \abs{f_{B_{i}} - f_{B_{i+1}}}2^{i\varepsilon}.
\]
Recall that this is quite similar to what we had in \eqref{eq:MT2} after applying the weighted Hölder's inequality.
Similar concept seems to also play a role in the \emph{$SD_p^s(w)$ weight condition} introduced in \cite[Definition 1.4]{Perez2018DegeneratePI}. In particular, the computations in \cite[Pages 6106-6107]{Perez2018DegeneratePI} seems to suggest this. See also the \emph{bumbed balance condition} introduced in \cite{kinnunenvahakangas2019}, and Section \ref{Sec:Applications} for further discussion.
\end{remark}

\begin{remark}\label{rem:dirac}
    The modified statement of Proposition \ref{prop:LocalUD}, where instead of $\delta > 0$ we have $\delta = 0$, is false in general. This in particular shows that Theorem \ref{thm:SobolevPoincare} is sharp in some sense.
    Consider the classical setting $W^{1,n}(\R^n)$ and let $\nu := \delta_0$ be the Dirac delta measure concentrated at the origin. Since $\Psi(x,r) = r^n \approx \abs{B(x,r)}$, where $\abs{A}$ denotes the Lebesgue measure, $\nu$ satisfies \eqref{eq:LocalUD} for $\delta = 0$. However, the Poincar\'e inequality in the claim of Proposition \ref{prop:LocalUD} where $x_0 = 0$ and $R_0 = 1$, fails. Even the weak type estimate
    \[
    \nu(\{ x \in B(0,1) : \abs{\tilde{f}(x) - f_{B(x_0,R_0)}} > s \}) \leq \frac{C}{s^q} \left( \int_{B(0,\sigma)} \abs{\nabla f}^p \,dx \right)^{\frac{q}{p}},
    \]
    which would be implied by the strong type counterpart and Chebyshev's inequality, is false. This can be seen from the fact that the precise representative $\tilde{f}$ of a Sobolev function in $f \in W^{1,n}(\R^n)$ might satisfy $\tilde{f}(0)=\infty$.
\end{remark}

We shall finish this subsection by proving the first implication of Theorem \ref{thm:SobolevPoincare}.

\begin{proof}[Proof of Theorem \ref{thm:SobolevPoincare}: \ref{T1} $\Rightarrow$ \ref{T2}]
    Fix $x \in X$ and $r \in (0,\diam(X))$.
    Let $y \in X$ and $s > 0$ be another pair such that
    \[
        B(y,s) \cap B(x,r) \neq \emptyset \text{ with } 0 < s \leq r.
    \]
    It follows from the condition \ref{T1} and the doubling property $\Theta$ \eqref{eq:Psi},
    \[
        \nu(B(x,s))^{\frac{1}{q}} \leq K\Theta(x,s)\left( \frac{\mu(B(x,s))}{\Psi(x,s)} \right)^{\frac{1}{p}} \leq C\Theta(x,r) \left( \frac{s}{r} \right)^\delta \left( \frac{\mu(B(x,s))}{\Psi(x,s)} \right)^{\frac{1}{p}}.
    \]
    By setting $L_0 := C\Theta(x,r)$ for a suitable $C \geq 1$, $x_0 := x$ and $R_0 := r$, the inequality \ref{T2} now follows from Proposition \ref{prop:LocalUD}.
\end{proof}

\subsection{Proof of Theorem \ref{thm:SobolevPoincare}; Part \textup{(II)}}
\label{subsec:converse}
Next, we cover the converse implication \ref{T2} $\Rightarrow$ \ref{T1}. This direction is much easier than the other.
Again, we shall verify a quantitative local estimate, and here we do not use \ref{eq:PI}.

\begin{proposition}\label{prop:Converse}
    Let $1 \leq p \leq q$, $(\mathcal{F}_p,\Gamma_p,\Psi)$ be a $p$-energy structure and $\nu$ be a Borel measure on $X$. Assume that the pair $x_0 \in X$ and $R_0 \in (0,\diam(X))$ satisfies the following condition. There are constants $L_0,\sigma \geq 1$ such that, whenever $x \in X$ and $r > 0$ with
    \begin{equation*}
        B(x,r) \cap B(x_0,R_0) \neq \emptyset \text{ and } 0 < r \leq R_0,
    \end{equation*}
    it holds for every $f \in \mathcal{F}_p$,
    \begin{equation}\label{eq:ConverseSPI}
    \left(\int_{B(x,r)} \abs{\tilde{f} - f_{B(x,r)}}^{q} \, d\nu \right)^{\frac{1}{q}} \leq L_0 \left(\int_{B(x,\sigma r)} d\Gamma_p\Span{f} \right)^{\frac{1}{p}}.
    \end{equation}
    Then there is a constant $C \geq 1$ such that
    \begin{equation}\label{eq:ConverseUD}
        \nu(B(x_0,R_0))^{\frac{1}{q}} \leq C L_0 \left( \frac{\mu(B(x_0,R_0))}{\Psi(x_0,R_0)} \right)^{\frac{1}{p}}.
    \end{equation}
    This result is quantitative in the sense that $C$ depends only on the constants in \ref{eq:UCap}, \eqref{eq:muDoubling}, \eqref{eq:RadDiam}, the constants associated to $\Psi$ in \eqref{eq:Psi}, $\sigma,\, p$ and $q$.
\end{proposition}

\begin{proof}
It follows from the doubling property of $\mu$ \eqref{eq:muDoubling} that there is a constant $\tau \geq \kappa$, where $\kappa$ is as in \ref{eq:UCap}, depending only on the doubling constant of $\mu$ and $\kappa$ satisfying the following condition. Whenever $x \in X$ and $r \in (0,R_0/\tau)$,
\begin{equation}\label{eq:sigmaDoubling}
    \frac{\mu(B(x,\kappa r))}{\mu(B(x,\tau r))} \leq \frac{1}{2}.
\end{equation}

Now, fix $x \in X$, $r \in (0,R_0/\tau)$ and let $\varphi \in \mathcal{F}_p$ be a cutoff function for $B(x,r) \subseteq B(x,\kappa r)$ provided by \ref{eq:UCap}.
Note that $\tilde{\varphi} = \varphi$ because $\varphi$ is continuous.
It follows from \eqref{eq:sigmaDoubling} and the definition cutoff functions that for all $y \in B(x,r)$,
\begin{equation*}
    \varphi(y) - \varphi_{B(x,\tau r)} = 1 - \mu(B(x,\tau r))^{-1}\int_{B(x,\kappa r)} \varphi \, d\mu \geq 1 - \frac{\mu(B(x,\kappa r))}{\mu(B(x,\tau r))} \geq \frac{1}{2}.
\end{equation*}
Since $\tau r \leq R_0$, we can use \eqref{eq:ConverseSPI} to get
\begin{align}\label{eq:EstimateSmallR}
    \nu(B(x,r))^{\frac{1}{q}} & \lesssim \left(\int_{B(x,\tau r)} \abs{\varphi - \varphi_{B(x,\tau r)}}^q \, d\nu\right)^{\frac{1}{q}}\\
    & \leq L_0 \Gamma_p\Span{\varphi}(X)^{\frac{1}{p}} \lesssim L_0 \left( \frac{\mu(B(x,r))}{\Psi(x,r)} \right)^{\frac{1}{p}}. \nonumber
\end{align}

Observe that the objective of the current proof is to obtain inequality in the previous display for $x = x_0$ and $r = R_0$, but currently we know this only when $r \in (0,R_0/\tau)$.
This is easily resolved by a covering argument.
Indeed, we take a covering $\{B(x_i,s)\}_{i \in I}$ of $B(x_0,R_0)$ such that $x_i \in B(x_0,R_0)$, $s := R_0/(2\tau)$ and  $d(x_i,x_j) \geq s$ for all distinct pairs $i,j \in I$. Then it follows from \eqref{eq:EstimateSmallR} and the doubling properties of $\mu$ and $\Psi$,
\[
    \nu(B(x_0,R_0))^{\frac{1}{q}} \leq \left( \sum_{i \in I} \nu(B(x_i,s)) \right)^{\frac{1}{q}} \lesssim \abs{I}^{\frac{1}{q}} L_0  \left( \frac{\mu(B(x_0,R_0))}{\Psi(x_0,R_0)} \right)^{\frac{1}{p}}.
\]
Here $\abs{I}$ is the cardinality of $I$.
By the metric doubling property of $X$, $\abs{I}$ is bounded above by a constant depending only on $N$ in \eqref{eq:MD} and $\tau$. This completes the proof.
\end{proof}

\begin{proof}[Proof of Theorem \ref{thm:SobolevPoincare}: \ref{T2} $\Rightarrow$ \ref{T1}]
    Fix $x \in X$ and $r \in (0,\diam(X))$.
    Let $y \in X$ and $s > 0$ be another pair such that
    \[
        B(y,s) \cap B(x,r) \neq \emptyset \text{ with } 0 < s \leq r.
    \]
    By the doubling property of  $\Theta$ \eqref{eq:Psi},
    \[
        \Theta(y,s) \lesssim \left( \frac{s}{r} \right)^\delta \Theta(x,r) \leq \Theta(x,r).
    \]
    By setting $L_0 := C \Theta(x,r)$ for suitable $C \geq 1$, by \ref{T2} we have for all $f \in \mathcal{F}_p$,
    \[
        \left(\int_{B(y,s)} \abs{\tilde{f} - f_{B(y,s)}}^{q} \, d\nu \right)^{\frac{1}{q}} \leq L_0 \left(\int_{B(y,\sigma s)} d\Gamma_p\Span{f} \right)^{\frac{1}{p}}.
    \]
    The inequality \ref{T1} now follows from Proposition \ref{prop:Converse} by choosing the parameters $x_0 := x$ and $R_0 := r$.
\end{proof}

\section{Dirichlet forms and heat kernel estimates}\label{sec:Dir}
The objective of this section is to prove Theorem \ref{thm:SUPER}. We also present a version of the result that applies to settings where the ambient space is not necessarily geodesic in Theorem \ref{thm:main}.

We impose the same assumption for the ambient space as in the previous section. Namely, $X = (X,d,\mu)$ is a metric measure space where $(X,d)$ is a complete, locally compact, and separable metric space. We assume that $\mu$ is a Radon measure and that it satisfies the doubling property \eqref{eq:muDoubling}. Lastly, we assume that $(X,d)$ is uniformly perfect, meaning \eqref{eq:RadDiam} holds.

\subsection{Definitions}
First, we fix the terminology on Dirichlet forms and refer to \cite{FOT,chen2012symmetric} for general background.

\begin{definition}\label{def:Dir}
We say that $(\mathcal{E},\mathcal{F})$ is a \emph{Dirichlet form} on $L^2(X,\mu)$ if the following two conditions hold.
    \begin{enumerate}
        \vspace{3pt}
        \item[(1)] $\mathcal{E}: \mathcal{F} \times \mathcal{F} \to \R$ is a symmetric non-negative definite bilinear form such that $\mathcal{F} \subseteq L^2(X,\mu)$ is a dense linear subspace, and $\mathcal{F}$ equipped with the inner product $\mathcal{E}_1(f,g) := \mathcal{E}(f,g) + \int_X f \cdot g \, d\mu$ is a Hilbert space.
        \vspace{3pt}
        \item[(2)] For all $f \in \mathcal{F}$ we have $f^+ \land 1 \in \mathcal{F}$ and $\mathcal{E}(f^+ \land 1,f^+ \land 1) \leq \mathcal{E}(f,f)$. This condition is called the \emph{Markov property}.
    \end{enumerate}
    We consider two additional conditions.
    \begin{enumerate}
        \vspace{3pt}
        \item[(3)] We say that a Dirichlet form $(\mathcal{E},\mathcal{F})$ on $L^2(X,\mu)$ is \emph{regular} if the subspace $\mathcal{F} \cap C_c(X)$ is dense in both the Hilbert space $(\mathcal{F},\mathcal{E}_1)$ and the normed space $(C_c(X),\norm{\cdot}_{L^\infty})$.
        \vspace{3pt}
        \item[(4)] We say that a Dirichlet form $(\mathcal{E},\mathcal{F})$ on $L^2(X,\mu)$ is \emph{strongly local} if the following implication always holds. Whenever $f,g \in \mathcal{F}$ such that their supports $\supp_\mu[f], \supp_\mu[g] \subseteq X$ are compact and there is $a \in \R$ such that $\supp_{\mu}[f] \cap \supp_{\mu}[g - a \mathds{1}_X] = \emptyset$, we have $\mathcal{E}(f,g) = 0$.
        Here $\mathds{1}_X$ denotes the constant function $x \mapsto 1$, and $\supp_{\mu}[f]$ is the smallest closed set $F \subseteq X$ with $\int_{X\setminus F} f d\mu = 0$.
    \end{enumerate}
    Given a strongly local regular Dirichlet form $(\mathcal{E},\mathcal{F})$ on $L^2(X,\mu)$, every $f \in \mathcal{F}$ is assigned the associated \emph{energy measure} $\Gamma\Span{f}$ as follows; see \cite[Chapter 3]{FOT} for details.
    If $f \in \mathcal{F}\cap L^\infty(X,\mu)$, then $\Gamma\Span{f}$ is the unique non-negative Radon measure on $(X,d)$ satisfying
    \[
        \int_X \varphi \, d\Gamma\Span{f} = \mathcal{E}(f,f\varphi) - \frac{1}{2}\mathcal{E}(f^2,\varphi) \quad \text{for all } \varphi \in \mathcal{F} \cap C_c(X).
    \]
    For a general $f \in \mathcal{F}$, we now define $\Gamma\Span{f}(A) = \lim_{k \to \infty} \Gamma\Span{(f \lor -k) \land k}(A)$.
\end{definition}

\begin{definition}
We say that $\Psi : (0,\infty) \to (0,\infty)$ is a \emph{radial scale function} if it is an increasing homeomorphism satisfying the following doubling type property for some $\beta_U,\beta_L > 1$. There is $C \geq 1$ such that for all $0 < r \leq R$,\begin{equation}\label{eq:PsiRad}C^{-1}\left(\frac{R}{r}\right)^{\beta_L} \leq \frac{\Psi(R)}{\Psi(r)} \leq C\left(\frac{R}{r}\right)^{\beta_U}.\end{equation}Given such $\Psi$, we associate it a function $\Phi$ given by\begin{equation}\label{eq:PhiRad}\Phi(s) := \sup_{r > 0} \left(\frac{s}{r} - \frac{1}{\Psi(r)}\right).\end{equation}\end{definition}

Throughout the rest of the section, we consider a fixed strongly local regular Dirichlet form $(\mathcal{E},\mathcal{F})$ on $L^2(X,\mu)$, a fixed radial scale function $\Psi$, and $\Phi$ is as in \eqref{eq:PhiRad}.

We consider the following more general version of sub-Gaussian heat kernel estimates than discussed in Introduction. 
Most notably, for the lower bounds of the heat kernel, we only require the so called \emph{near diagonal estimates}.
If the ambient space is geodesic, then a chaining argument can be used to establish the full lower bound from the near diagonal variant;
see \cite[Page 1217]{grigor2012two} for further discussion on these differences.

\begin{definition}\label{def:heatkernel}
    Let $(\mathcal{E},\mathcal{F})$ be a strongly local regular Dirichlet form on $L^2(X,\mu)$ and $\{P_t\}_{t > 0}$ be the associated Markov semigroup; see \cite[Section 1.4]{FOT}.
    A family of Borel measurable function $\{p_t\}_{t > 0}$, $p_t : X \times X \to [0,\infty]$, is a \emph{heat kernel} of $(\mathcal{E},\mathcal{F})$ if for every $t > 0$ the function $p_t$ is an integral kernel of $P_t$, meaning for all $t > 0$ and $f \in L^2(X,\mu)$,
    \[
    P_t(f)(x) = \int_X p_t(x,y)f(y)\, d\mu \quad \text{for $\mu$-almost every $x \in X$.}
    \]
    We say that $(\mathcal{E},\mathcal{F})$ satisfies the \emph{heat kernel estimates} \ref{eq:HKMain} if there is a heat kernel $\{p_t\}_{t > 0}$ of $(\mathcal{E},\mathcal{F})$ and constants $C,C_1,C_2,c,\kappa > 0$ such that for all $t > 0$,
    \begin{align}\label{eq:HKMain}
    \tag{$\textup{HKE}(\Psi)$}
    p_t(x,y) & \leq \frac{C}{\mu(B(x,\Psi^{-1}(t))}\exp\left( -C_1t\Phi \left(C_2\frac{d(x,y)}{t}\right) \right) \text{ for $\mu$-almost every $x,y \in X$}, \\
    p_t(x,y) & \geq \frac{c}{\mu(B(x,\Psi^{-1}(t)))} \text{ for $\mu$-almost every $x,y \in X$ with $d(x,y) \leq \kappa \Psi^{-1}(t)$}. \nonumber
\end{align}
    If $\Psi(r) = r^\beta$ for some $\beta \geq 2$, we write \hyperref[eq:HKMain]{$\textup{HKE}(\beta)$} instead.
\end{definition}

\begin{definition}\label{def:PI(Dir)}
    We say that $(\mathcal{E},\mathcal{F})$ satisfies the \emph{Poincar\'e inequality} \ref{eq:PI(Dir)} if there are constants $C,\sigma \geq 1$ satisfying the following condition.
    For all $x \in X$, $r > 0$ and $f \in \mathcal{F}$,
    \begin{equation}\label{eq:PI(Dir)}
        \tag{$\textup{PI}(\Psi)$}
        \int_{B(x,r)} (f - f_{B(x,r)})^2 \, d\mu \leq C \Psi(r) \int_{B(x,\sigma r)} \, d\Gamma\Span{f}.
    \end{equation}
    We also say that $(\mathcal{E},\mathcal{F})$ satisfies the \emph{upper capacity estimate} \ref{eq:Ucap(Dir)} if there is a constant $C \geq 1$ satisfying the following condition. For all $x \in X$ and $r > 0$ there is a cutoff function $\varphi \in \mathcal{F}$ for $B(x,r) \subseteq B(x,2r)$ such that
    \begin{equation}\label{eq:Ucap(Dir)}
        \tag{$\textup{Cap}_\leq(\Psi)$}
        \mathcal{E}(\varphi) \leq C \frac{\mu(B(x,r))}{\Psi(r)}.
    \end{equation}
    If $\Psi(r) = r^\beta$ for some $\beta \geq 2$, we write \hyperref[eq:PI(Dir)]{$\textup{PI}(\beta)$} and \hyperref[eq:Ucap(Dir)]{$\textup{Cap}_\leq(\beta)$} instead.
\end{definition}

The following definition is a slightly modified variant of the cutoff Sobolev inequality introduced in \cite{barlow2004stability,barlow2006stability}.
The main difference is that we have removed the Hölder regularity condition because, as we prove in Corollary \ref{cor:holder}, it actually follows from the other conditions we consider.
It has been also removed from some later variants \cite{andres2015energy}.

\begin{definition}\label{def:CS(Dir)}
    For a given $\delta > 0$, we say that $(\mathcal{E},\mathcal{F})$ satisfies the \emph{cutoff Sobolev inequality} \ref{eq:CS(Dir)} if there is a constant $C \geq 1$ for which the following holds.
    For every $x \in X$ and $R \in (0,\diam(X))$ there is a cutoff function $\xi \in \mathcal{F}_p$ for $B(x,R) \subseteq B(x,2R)$ such that
    \begin{equation}\label{eq:CS(Dir)}
    \tag{$\textup{CS}_{\delta}(\Psi)$}
    \int_{B(y,r)} f_{qc}^2 \, d\Gamma\Span{\xi} \leq C \left(\frac{r}{R}\right)^\delta \left( \int_{B(y,2r)} \, d\Gamma\Span{f} + \frac{1}{\Psi(r)} \int_{B(y,2r)} f^2 \, d\mu \right). \end{equation}
    for all $y \in X$, $0 < r \leq 3R$ and $f \in \mathcal{F}$.
    Here $f_{qc}$ is any quasicontinuous $\mu$-representative of $f$; see \cite[Chapter 2]{FOT} for a detailed exposition on quasicontinuity.
    If $\Psi(r) = r^\beta$ for some $\beta \geq 2$, we write \hyperref[eq:CS(Dir)]{$\textup{CS}_\delta(\beta)$} instead.
\end{definition}

Lastly, we define the cutoff energy condition.

\begin{definition}
     For a given $\delta > 0$, we say that $(\mathcal{E},\mathcal{F})$ satisfies the \emph{cutoff energy condition} \ref{eq:CE(Dir)} if there is a constant $C \geq 1$ satisfying the following. For all $x \in X$ and $R \in (0,\diam(X))$ there is a cutoff function $\xi \in \mathcal{F}$ for $B(x,r) \subseteq B(x,2r)$ such that
    \begin{equation}\label{eq:CE(Dir)}
        \tag{$\textup{CE}_{\delta}(\Psi)$}
        \int_{B(y,r)} \,d\Gamma\Span{\xi} \leq C\left( \frac{r}{R} \right)^\delta \frac{\mu(B(y,r))}{\Psi(r)}
    \end{equation}
    for all $y \in X$ and $0 < r \leq 3R$.
    If $\Psi(r) = r^\beta$ for some $\beta \geq 2$, we write \hyperref[eq:CE(Dir)]{$\textup{CS}_\delta(\beta)$} instead.
\end{definition}

\subsection{Main theorem}
Now, we are in the position of proving the main theorem of the work.
We begin by stating it in a more general form.
Note that the doubling property of $\mu$ is now an assumption rather than a condition in the equivalence.

\begin{theorem}\label{thm:main}
    The strongly local regular Dirichlet form $(\mathcal{E},\mathcal{F})$ on $L^2(X,\mu)$ satisfies the heat kernel estimates \ref{eq:HKMain} if and only if $(\mathcal{E},\mathcal{F})$ satisfies both the Poincar\'e inequality \ref{eq:PI(Dir)} and the cutoff energy condition \ref{eq:CE(Dir)} for some $\delta > 0$.
\end{theorem}

The first part of our argument is the following equivalence between \ref{eq:CS(Dir)} and \ref{eq:CE(Dir)}.

\begin{proposition}\label{prop:CSandCE}
    Let $\delta > 0$ and assume that $(\mathcal{E},\mathcal{F})$ satisfies \ref{eq:PI(Dir)}.
    Then $(\mathcal{E},\mathcal{F})$ satisfies \ref{eq:CS(Dir)} if and only if it satisfies \ref{eq:CE(Dir)}.
\end{proposition}

\begin{remark}
    Under the assumptions in Proposition \ref{prop:CSandCE}, $(\mathcal{F},\Gamma,\Psi)$ is a 2-energy structure in the sense of Definition \ref{def:SSS}.
    The upper capacity estimate \hyperref[eq:PI]{$\textup{Cap}_{2,\leq}(\Psi)$} is not needed here but it is implied by \ref{eq:CE(Dir)}/\ref{eq:CS(Dir)} and the strong locality.
\end{remark}

\begin{proof}[Proof of Proposition \ref{prop:CSandCE}]
    The first implication \ref{eq:CS(Dir)} $\Rightarrow$ \ref{eq:CE(Dir)} is direct. Fix $x \in X$ and $R \in (0,\diam(X))$, and let $\xi \in \mathcal{F}$ be a cutoff function for $B(x,R) \subseteq B(x,2R)$ provided by \ref{eq:CS(Dir)}. Fix any $y \in X$ and $0 < r \leq 3R$, and choose $h \in \mathcal{F} \cap C(X)$ such that $h|_{B(y,2r)} = 1$. Such function $h$ exists by \cite[Exercise 1.4.1]{FOT}, or also by \ref{eq:CS(Dir)}.
    Then \ref{eq:CE(Dir)} follows by applying \ref{eq:CS(Dir)} to $h$,
    \begin{align*}
        \int_{B(y,r)} \,d\Gamma\Span{\xi} & = \int_{B(y,r)} h^2 \,d\Gamma\Span{\xi} \lesssim \left(\frac{r}{R}\right)^\delta \left( \int_{B(y,2r)} \, d\Gamma\Span{h} + \frac{1}{\Psi(r)} \int_{B(y,2r)} h^2 \, d\mu \right)\\
        & = \left(\frac{r}{R}\right)^\delta \frac{\mu(B(x,2r))}{r^\beta} \lesssim \left(\frac{r}{R}\right)^\delta \frac{\mu(B(x,r))}{\Psi(r)}.
    \end{align*}
    Here we used $\Gamma\Span{h}(B(x,2r)) = 0$ which holds by the strong locality of the Dirichlet form, Definition \ref{def:Dir}-(4), and $h|_{B(x,2r)} = 1$; see \cite[Corollary 3.2.1]{FOT}.

    Next, we prove the converse direction \ref{eq:CE(Dir)} $\Rightarrow$ \ref{eq:CS(Dir)}.
    Here we need to be careful with the different $\mu$-representatives because the results in Section \ref{sec:SPI} regard the precise representatives whereas \ref{eq:CS(Dir)} is formulated in terms quasicontinuous representatives.
    Therefore, we first consider a continuous function $f \in \mathcal{F} \cap C(X)$ so that $f = \tilde{f} = f_{qc}$ holds pointwise.
    
    Again, fix $x \in X$ and $R \in (0,\diam(X))$, and let $\xi \in \mathcal{F}$ be a cutoff function for $B(x,R) \subseteq B(x,2R)$ provided by \ref{eq:CE(Dir)}.
    We check the required conditions to apply Proposition \ref{prop:LocalUD} for $\nu = \Gamma\Span{\xi}$.
    Fix $y \in X,\, 0 < r \leq 3R$ and $z \in X$, $0 < s \leq r$ such that $B(y,r) \cap B(z,s) \neq \emptyset$.
    By the cutoff energy condition \ref{eq:CE(Dir)},
    \[
        \int_{B(z,s)} \, d\Gamma\Span{\xi} \leq  C\left(\frac{s}{R}\right)^\delta \frac{\mu(B(z,s))}{\Psi(s)} = L_0 \left(\frac{s}{r}\right)^\delta \frac{\mu(B(z,s))}{\Psi(s)}
    \]
    where $L_0 := C (r/R)^\delta$.
    Thus, we have verified the conditions in Proposition \ref{prop:LocalUD} for the parameters $\nu = \Gamma\Span{\xi}$, $p = 2 = q$, $x_0 = y$ and $R_0 = r$, and we obtain the Poincar\'e type inequality
    \begin{equation}\label{eq:CEY}
        \int_{B(y,r)} (f - f_{B(y,r)})^2 \, d\Gamma\Span{\xi} \lesssim \left( \frac{r}{R} \right)^\delta \int_{B(y,\sigma r)} \,d\Gamma\Span{f}.
    \end{equation}
    We recall that the average $f_{B(y,r)}$ is taken with respect to the reference measure $\mu$.
    
    Now, we estimate
    \begin{align*}
        &  \, \int_{B(y,r)} f^2 \, d\Gamma\Span{\xi}\\
        \lesssim \quad & \int_{B(y,r)} (f - f_{B(y,r)})^2 d\Gamma\Span{\xi} + \Gamma\Span{\xi}(B(y,r)) \kint_{B(y,r)} f^2 \,d\mu && (\text{Jensen's ineq.}) \\
        \lesssim \quad &\left(\frac{r}{R}\right)^\delta \int_{B(y,\sigma r)} \, d\Gamma\Span{f} + \Gamma\Span{\xi}(B(y,r)) \kint_{B(y,r)} f^2 \,d\mu && (\text{Equation \eqref{eq:CEY}}) \\
        \lesssim \quad & \left(\frac{r}{R}\right)^\delta \left( \int_{B(y,\sigma r)} \, d\Gamma\Span{f} + \frac{1}{\Psi(r)} \int_{B(y,\sigma r)} f^2 \, d\mu \right). && \text{\eqref{eq:CE(Dir)}}
    \end{align*}
    In the second line, we also used $(a + b)^2 \leq 2(a^2 + b^2)$.
    Thus, we get
    \[
        \int_{B(y,r)} f^2 \, d\Gamma\Span{\xi} \leq C\left(\frac{r}{R}\right)^\delta \left( \int_{B(y,\sigma r)} \, d\Gamma\Span{f} + \frac{1}{\Psi(r)} \int_{B(y,\sigma r)} f^2 \, d\mu \right).
    \]
    Observe that the inequality in the previous display is almost the cutoff Sobolev inequality \ref{eq:CS(Dir)} and the only difference is that we could have $\sigma > 2$. However, by using a similar covering argument as in the proof of Proposition \ref{prop:Converse}, we get the previous inequality for $\sigma = 2$, namely \ref{eq:CS(Dir)}, for possibly larger constant $C$. This concludes the case when $f$ is continuous.

    Next, we consider a general $f \in \mathcal{F}$, and let $f_{qc}$ be any quasicontinuous $\mu$-representative of $f$. By the regularity, Definition \ref{def:Dir}-(3), there is a sequence $\{f_n\}_{n = 1}^\infty \subseteq \mathcal{F} \cap C(X)$ such that $f_n \to f$ in the Hilbert space $(\mathcal{F},\mathcal{E}_1)$.
    It then follows from \cite[Theorem 2.1.4 and Lemma 3.2.4]{FOT} that, by taking a subsequence if necessary, $f_n$ converges to $f$ point-wise $\Gamma\Span{\xi}$-almost everywhere. Since the right-hand side of \ref{eq:CS(Dir)} is continuous with respect to $\mathcal{E}_1$, and we have verified \ref{eq:CS(Dir)} for continuous $f$, the general case $f \in \mathcal{F}$ now follows from Fatou's lemma.
\end{proof}

\begin{remark}\label{rem:qc}
    Even under the assumptions of Proposition \ref{prop:CSandCE}, we do not know whether the precise representative $\tilde{f}$ in \eqref{eq:Exact} is quasicontinuous. If it was, then the approximation argument in the previous proof is unnecessary. In \cite{kinnunen2002lebesgue} Kinnunen and Latvala resolved this issue in a certain setting of analysis on metric spaces.
    They used a certain capacitary weak type estimate of maximal function and, unfortunately, we do not know when this technique can be generalized to Dirichlet forms.
\end{remark}

Next, we prove the implication \ref{eq:HKMain} $\Rightarrow$ \ref{eq:CE(Dir)}.
The method is essentially the same as the proof of \ref{eq:HKMain} $\Rightarrow$ \ref{eq:CS(Dir)} in  Barlow--Bass--Kumagai \cite{barlow2006stability} but we nevertheless include the details since therein setting slightly differs from the present one\footnote{Some related details were corrected in a later version which is available in arXiv \cite{barlow2006stability}.}. Moreover, we have isolated the details required for the cutoff energy condition.

\begin{proposition}\label{prop:BBK}
    If the strongly local regular Dirichlet form $(\mathcal{E},\mathcal{F})$ satisfies \ref{eq:HKMain} then it satisfies \ref{eq:CE(Dir)} for some $\delta > 0$.
\end{proposition}

The following lemma is used in the argument.

\begin{lemma}[Theorem 1.2 \cite{GrigorCapacity15}]
\label{lemma:Ucap}
    If the strongly local regular Dirichlet form $(\mathcal{E},\mathcal{F})$ on $L^2(X,\mu)$ satisfies \ref{eq:HKMain} then it satisfies \ref{eq:PI(Dir)} and \ref{eq:Ucap(Dir)}.
\end{lemma}

The proof of Proposition \ref{prop:BBK} is based on the usage of resolvents; for background see \cite[Chapter 1]{FOT}. Given $\lambda > 0$, we define the \emph{$\lambda$-resolvent} as the operator
\begin{equation*}G_\lambda f(x) := \int_X \int_0^\infty e^{-\lambda t} p_t(x,y) f(y)\, dt \, d\mu(y) \end{equation*} where $\{p_t\}_{t > 0}$ is a heat kernel of $(\mathcal{E},\mathcal{F})$ satisfying \ref{eq:HKMain}.

Given $x_0\in X$ and $R_0 > 0$, we first study the function $h_{x_0,R_0} := G_{\lambda} \varphi$ where $\varphi_{x_0,R_0} \in \mathcal{F}$ is a cutoff function for $B(x_0,R_0\kappa /16) \subseteq B(x_0,R_0\kappa /8)$ contained in the domain of the resolvents, and $\lambda := \Psi(R_0)^{-1}$ and $\kappa$ is as in \ref{eq:HKMain}.
It follows from the basic properties of $G_\lambda$ that $h_{x_0,R_0} \in \mathcal{F}$ and, for all $g \in \mathcal{F}$,
\begin{equation}\label{eq:ResolventEq}
    \mathcal{E}(h_{x_0,R_0},g) = \int_X \varphi_{x_0,R_0} \cdot g \, d\mu - \Psi(R_0)^{-1} \int_X h_{x_0,R_0} \cdot g \, d\mu.
\end{equation}

\begin{lemma}\label{lemma:K}
   If the strongly local regular Dirichlet form $(\mathcal{E},\mathcal{F})$ on $L^2(X,\mu)$ satisfies \ref{eq:HKMain} then there are constants $K \geq 1$ and $\sigma > 1$ depending only on the constants associated to $\Psi$ in \eqref{eq:PsiRad}, the constants in \ref{eq:HKMain} and the doubling constant of $\mu$ in \eqref{eq:muDoubling} such that
    \begin{align*}
        h_{x_0,R_0}(x) & \geq 2K \Psi(R_0) \text{ for all } x \in B(x_0,\kappa R_0),\\
        h_{x_0,R_0}(x) & \leq K \Psi(R_0) \text{ for all } x \in X \setminus B(x_0,\sigma R_0).
    \end{align*}
    Here $\kappa$ is the constant in \ref{eq:HKMain}.
\end{lemma}

\begin{proof}
    Let $C,C_1,C_2, c,\kappa$ be as in \ref{eq:HKMain}.
    First, for $x \in B(x_0,\kappa R_0)$, we use the lower bound in \eqref{eq:HKMain},
    \begin{align*}
        h_{x_0,R_0}(x) & \geq \int_{B(x_0,\kappa R_0/16)} \int_{\Psi(2\kappa R_0)}^{\Psi(4 \kappa R_0)} e^{-t/\Psi(R_0)} p_{t}(x,y) \,dt\, d\mu(y)\\
        & \geq c\int_{B(x_0,\kappa R_0/16)} \int_{\Psi(2\kappa R_0)}^{\Psi(4 \kappa R_0)}e^{-t/\Psi(R_0)} \frac{1}{\mu(B(x,\Psi^{-1}(t)} \,dt\, d\mu(y)\\
        & \geq  c \frac{\mu(B(R_0\kappa /16))}{\mu(B(x,R_0 4\kappa))} \int_{\Psi(2\kappa R_0)}^{\Psi(4 \kappa R_0)} e^{-t/\Psi(R_0)} \,dt\\
        & = c \left(e^{-\Psi(2\kappa R_0)/\Psi(R_0)} - e^{-\Psi(4\kappa R_0)/\Psi(R_0)}\right) \frac{\mu(B(R_0\kappa/16))}{\mu(B(x,4\kappa R_0))} \Psi(R_0)\\
        & \geq 2K \Psi(R_0),
    \end{align*}
    where $K$ depends on the constants in \eqref{eq:PsiRad} and the doubling constant of $\mu$ in \eqref{eq:muDoubling}.
    
    Next, fix a constant $\tau > 1$ and let $x \in X \setminus B(x_0,2 \tau \kappa R_0)$. Then, by using the upper bound in \ref{eq:HKMain},
    \begin{align*}
        & \quad h_{x_0,R_0}(x)\\
        = & \int_X \int_0^\infty e^{-t/\Psi(R_0)} p_t(x,y) \varphi_{x_0,R_0} \, dt \, d\mu\\
        \leq & \int_{B(x_0,\kappa R_0/8)} \int_{0}^\infty e^{-t/\Psi(R_0)} \frac{C}{\mu(B(x,\Psi^{-1}(t)}\exp\left( - C_1 t \Phi\left(C_2\frac{d(x,y)}{t}\right) \right) \, dt\, d\mu(y)\\
        \leq &\, C\mu(B(x_0,\kappa R_0/8)) \int_{0}^\infty e^{-t/\Psi(R_0)} \frac{1}{\mu(B(x,\Psi^{-1}(t))} \exp\left(-C_1 t \Phi\left(\frac{2C_2\tau\kappa R_0}{t}\right)\right) \, dt.
    \end{align*}
    In the last inequality, we used the fact that $\Phi$ is, by definition, non-decreasing.
    
    We now divide the previous integral into two parts. First, we use the properties of the function $\Phi$ in \cite[Page 30]{murugan2024heat}, the doubling properties of $\mu$ \eqref{eq:RevDoubling} and $\Psi$ \eqref{eq:PsiRad},
    \begin{align}\label{eq:HKUBs}
        & \quad \int_{0}^{\Psi(\tau R_0)} e^{-t/\Psi(R_0)} \frac{1}{\mu(B(x,\Psi^{-1}(t))} \exp\left( - C_1 t \Phi\left(\frac{2C_2\tau \kappa R_0}{t}\right)
        \right) \, dt\\
       \leq & \quad \int_{0}^{\Psi(\tau R_0)} e^{-t/\Psi(R_0)} \frac{1}{\mu(B(x,\Psi^{-1}(t))} \exp\left( - C_3 \left( \frac{\Psi(\tau R_0)}{t}\right)^\frac{1}{\beta_L -1}
        \right) \, dt \nonumber \\
        \leq & \quad \frac{C_4}{\mu(B(x,\tau R_0))}\int_{0}^{\Psi(\tau R_0)} e^{-t/\Psi(R_0)}\left(\frac{\tau R_0}{\Psi^{-1}(t)}\right)^{Q_U}\exp\left( - C_3 \left( \frac{\Psi(\tau R_0)}{t}\right)^\frac{1}{\beta_L -1}
        \right) \, dt \nonumber \\
        \leq & \quad \frac{C_4}{\mu(B(x,\tau R_0))}\int_{0}^{\Psi(\tau R_0)} e^{-t/\Psi(R_0)}\left(\frac{\Psi(\tau R_0)}{t}\right)^{Q_U/\beta_L}\exp\left( - C_3 \left( \frac{\Psi(\tau R_0)}{t}\right)^\frac{1}{\beta_L -1}
        \right) \, dt \nonumber \\
        \leq & \quad \frac{A(\mu,\Psi)}{\mu(B(x,\tau R_0))}\int_{0}^{\Psi(\tau R_0)} e^{-t/\Psi(R_0)}\, dt \leq \frac{A(\mu,\Psi)}{\mu(B(x,\tau R_0))} \Psi(R_0).\nonumber
    \end{align}
    Here $Q_U$ is an upper exponent of $\mu$ and $\beta_L > 1$ is a lower exponent of $\Psi$.
    In between the second last row and the last row, we used $s^\alpha \leq C_\alpha \exp(s)$ for all $s \geq 1$ and $\alpha > 0$.
    The remaining part of the integral is estimated
    \begin{align*}
        & \quad \int_{\Psi(\tau R_0)}^\infty e^{-t/\Psi(R_0)} \frac{1}{\mu(B(x,\Psi^{-1}(t))} \exp\left(-C_1 t \Phi\left(\frac{2C_2\tau\kappa R_0}{t}\right)\right) \, dt\\
       \leq & \quad \frac{1}{\mu(B(x,\tau R_0))} \int_{\Psi(\tau R_0)}^\infty e^{-t/\Psi(R_0)}\exp\left(-C_3 t\Phi\left(\frac{\tau R_0}{t}\right)\right) \, dt\\
       \leq & \quad \frac{C_5}{\mu(B(x,\tau R_0))}\int_{\Psi(\tau R_0)}^\infty e^{-t/\Psi(R_0)} \, dt \leq \frac{C_5}{\mu(B(x,\tau R_0))}\Psi(R_0).
    \end{align*}
    Here, we also used some properties of $\Phi$ \cite[(4.8) in Page 30]{murugan2024heat}.
    By combining the previous three displays,
    \[
    h_{x_0,R_0}(x) \leq L \frac{\mu(B(x_0,\kappa R_0/8))}{\mu(B(x_0,\tau R_0))} \Psi(R_0),
    \]
    where $L$ depends only on the constants mentioned in the claim. Thus, by \eqref{eq:RevDoubling} and the first display in the proof, we reach the conclusion by setting $\sigma := 2\tau \kappa$ where $\tau$ is chosen to be suitably large.
\end{proof}

We also need a suitable local Hölder regularity of $h_{x_0,R_0}$.

\begin{lemma}\label{lemma:h}
    Let $\kappa,\sigma$ and $h_{x_0,R_0}$ be as in Lemma \ref{lemma:K}.
    There are constants $C \geq 1$ and $\delta > 0$ depending only on the constants in \ref{eq:HKMain}, the constants associated with $\Psi$ in \eqref{eq:PsiRad} and the doubling constant of $\mu$ in \eqref{eq:muDoubling} such that the following holds. For all $x \in B(x_0,2\sigma R_0)\setminus B(x_0,R_0\kappa/2)$ and $y,z \in B(x,R_0\kappa /16)$,
    \begin{equation}
        \abs{h_{x_0,R_0}(y)-h_{x_0,R_0}(z)} \leq C \left(\frac{d(y,z)}{R_0}\right)^{\delta}\Psi(R_0).
    \end{equation}
    
\end{lemma}

\begin{proof}
First, for the sake of convenience, we define the heat kernel also for non-positive times by $p_t(x,y) = 0$ for all $t \leq 0$ and $x,y \in X$.

Now, fix $x\in B(x_0,2\sigma R_0)\setminus B(x_0,R_0\kappa/2)$.
Then, given $w \in B(x_0,R_0 \kappa /8)$, the function $(y,t) \mapsto p_t(w,y)$ is caloric (in the sense of \cite{barlow2012equivalence}) in $X \setminus B(x,R_0 \kappa/7) \times \R$.
To see this, it is trivial when $t < 0$, and the case $t>0$ follows from the fact that the heat kernel is caloric. The case $t = 0$ follows from the upper bound in \ref{eq:HKMain} and the fact that $B(x,R_0 \kappa/6) \cap B(x_0,R_0\kappa /8) =\emptyset$.

Thus, it follows from the parabolic Harnack inequality \cite[Corollary 4.2]{barlow2012equivalence} that there are $C \geq 1$ and $\delta > 0$ such that for all $w \in B(x_0,R_0\kappa/8)$, $y,z \in B(x,R_0\kappa/16)$ and $t > 0$,
\[
    \abs{p_t(w,y) - p_t(w,z)} \lesssim \left(\frac{d(y,z)}{R_0}\right)^{\delta/2} \esssup_{(y,t)} p_t(w,y)
\]
where the essential supremum (with respect to $d\mu \otimes dt$) is over $(X \setminus B(x_0,R_0 \kappa /6)) \times \R$.
Using the fact the points $y$ in the previous $\esssup$ satisfy $d(w,y) \geq r \kappa /24$ along with the upper bound estimates used in \eqref{eq:HKUBs} we see that
\begin{align*}
    p_t(w,y) \lesssim \frac{1}{\mu(B(x_0,R_0)}.
\end{align*}
By integrating,
\begin{align*}
    \abs{h_{x_0,R_0}(y) - h_{x_0,R_0}(z)} & \leq \int_0^\infty
    \int_{B(x_0,R_0 \kappa /8)} e^{-t/\Psi(R_0)}\abs{p_t(w,y)-p_t(w,z)} \, d\mu(x) \, dt\\
    & \lesssim \left( \frac{d(y,z)}{R_0}\right)^{\delta/2} \int_0^\infty e^{-t/\Psi(R_0)} \, dt = \left( \frac{d(y,z)}{R_0}\right)^{\delta/2} \Psi(R_0),
\end{align*}
which completes the proof.
\end{proof}

We are now ready to establish Proposition \ref{prop:BBK}.
During the proof, we shall denote the two variable energy measures $\Gamma\Span{f,g} := 1/4(\Gamma\Span{f + g} - \Gamma\Span{f-g})$. Moreover, we also use some of their basic formulas and inequalities which can all be found in \cite[Pages 1488-1489]{GrigorCapacity15}.

\begin{proof}[Proof of Proposition \ref{prop:BBK}]
    Let $\kappa,\sigma,h_{x_0,R_0}$ and $K$ be as in Lemma \ref{lemma:K} and $\delta$ as in Lemma \ref{lemma:h}
    We define the cutoff function $\xi := (h_{x_0,R_0}\cdot (K\Psi(R_0))^{-1} - 1)^+ \land 1$, which is a cutoff function for $B(x_0,\kappa R_0) \subseteq B(x_0,\sigma R_0)$ by Lemma \ref{lemma:K}. By a simple covering argument and the strong locality, it is sufficient to show that
    \begin{equation}\label{eq:ConstXi}
        \int_{B(x,r)} \, d\Gamma\Span{\xi} \lesssim \left( \frac{r}{R_0} \right)^{\delta} \frac{\mu(B(x,r))}{\Psi(r)}
    \end{equation}
    for all $x \in B(x_0,2\sigma R_0)$ and $0 < r \leq R_0\kappa/16$.

    We fix such $x$ and $r$, and assume first that $d(x_0,x) < R_0\kappa/2$. It follows from $\xi|_{B(x_0,R_0\kappa)} = 1$ and the strong locality that $\Gamma\Span{\xi}(B(x,r)) = 0$, and therefore \eqref{eq:ConstXi} holds. We next assume $d(x_0,x) \geq R_0\kappa/2$.
    
    Since \ref{eq:Ucap(Dir)} holds by Lemma \ref{lemma:Ucap}, there is a cutoff function $f \in \mathcal{F}$ for $B(x,r) \subseteq B(x,2r)$ such that $\mathcal{E}(f) \lesssim \mu(B(x,r))/\Psi(r)$.
    Consider the function $\xi_0 := h_{x_0,R_0}\cdot (K\Psi(R_0))^{-1}$ and note that $\Gamma\Span{\xi}\leq \Gamma\Span{\xi_0}$.
    We let $\hat{\xi}_0 := \xi_0 - \inf_{B(x,4r)}\xi_0$. Then, by the definition of the energy measures and strong locality,

    \begin{align*} \int_{B(x,r)} \,d\Gamma\Span{\xi_0} & = \int_{B(x,r)} f \,d\Gamma\Span{\hat{\xi}_0,\xi_0} \leq \mathcal{E}(f\hat{\xi}_0,\xi_0) - \int_{B(x,2r)} \hat{\xi}_0 \, d\Gamma\Span{f,\xi_0}. \end{align*}
    For the two terms in the above display, the first one is estimated
\begin{align*}
K\Psi(R_0)\mathcal{E}(f\hat{\xi}_0,\xi_0) & \leq \mathcal{E}(f\hat{\xi}_0,h_{x_0,R_0}) + \Psi(R_0)^{-1}\int_X f \hat{\xi}_0h_{x_0,R_0}\, d\mu && (f \hat{\xi}_0h_{x_0,R_0} \geq 0) \\& = \Psi(R_0)^{-1} \int_{X} \varphi_{x_0,R_0} f\hat{\xi}_0 d\mu && (\text{Eq. \eqref{eq:ResolventEq}})  \\& = 0 && (f|_{B(x,R_0 \kappa/8)} = 0).\end{align*}
For the second term,
\begin{align*}
& \quad \left|\int_{B(x,2r)} \hat{\xi}_0 \,d\Gamma\Span{f,\xi_0} \right|\lesssim \left( \frac{r}{R} \right)^{\delta} \left|\int_{B(x,2r)}\, d\Gamma\Span{f,\xi_0} \right| && (\text{Lemma \ref{lemma:h}})\\
& \leq \quad \left( \frac{r}{R} \right)^{\delta} \left( \int_{B(x,2r)} \,d\Gamma\Span{f}\right)^{\frac{1}{2}} \left( \int_{B(x,2r)} \,d\Gamma\Span{\xi_0} \right)^{\frac{1}{2}} && \text{(Cauchy--Schwarz)}\\
& \lesssim \quad \left( \frac{r}{R} \right)^{\delta} \left( \frac{\mu(B(x,r))}{\Psi(r)}\right)^{\frac{1}{2}} \left( \int_{B(x,2r)} \,d\Gamma\Span{\xi_0} \right)^{\frac{1}{2}}. && \eqref{eq:Ucap(Dir)}
\end{align*}
Then, we use the fact that $\xi_0$ is subharmonic in $B(x,4r)$; see \cite[Page 1490]{GrigorCapacity15} for the definition.
Indeed, given a non-negative $\psi \in \mathcal{F}\cap C(X)$ such that $\psi|_{X\setminus B(x,4r)}=0$,
\begin{align*}
    K\Psi(R_0)\mathcal{E}(\xi_0,\psi) &= \mathcal{E}(h_{x_0,R_0},\psi)\\
    & = \mathcal{E}(h_{x_0,R_0},\psi)+\int_{X}\varphi_{x_0,R_0}\psi\, d\mu && (\varphi_{x_0,R_0} \cdot \psi = 0)\\
    & = -\Psi(R_0)^{-1} \int_{X}h_{x_0,R_0}\cdot \psi \,d\mu && (\text{Eq. \eqref{eq:ResolventEq}})\\
    & \leq 0. && (\psi \geq 0)
\end{align*}
Also note that it follows from the arguments in Lemma \ref{lemma:h} that $0 < \xi_0 \leq L$ in $B(x,2r)$ where $L$ depends only on the constants mentioned in Lemma \ref{lemma:h}.
We now use the log-Caccioppoli inequality \cite[Lemma 7.1]{GrigorCapacity15} and the chain rule,
\begin{align*}
    \int_{B(x,2r)}\, d\Gamma\Span{\xi_0} &= \int_{B(x,2r)} (\xi_0)^2 d\Gamma\Span{\log(\xi_0)} \lesssim L\frac{\mu(B(x,r)}{\Psi(r)}.
\end{align*}
By combining all the previous estimates,
\begin{align*}
    \int_{B(x,r)}\, d\Gamma\Span{\xi} \leq \int_{B(x,r)}\, d\Gamma\Span{\xi_0} \lesssim \left(\frac{r}{R_0}\right)^{\delta}\frac{\mu(B(x,r))}{\Psi(r)}.
\end{align*}
Thus, we have verified \eqref{eq:ConstXi} which completes the proof.
\end{proof}

We now have gathered everything we need for the main result of the work.

\begin{proof}[Proof of Theorem \ref{thm:main}]
    The first implication \ref{eq:HKMain} $\Rightarrow$ \ref{eq:PI(Dir)} \& \ref{eq:CE(Dir)} follows from Proposition \ref{prop:BBK} and Lemma \ref{lemma:Ucap}.
    
    The converse implication \ref{eq:PI(Dir)} \& \ref{eq:CE(Dir)} $\Rightarrow$ \ref{eq:HKMain} follows from the characterization of sub-Gaussian heat kernel estimates in \cite[Theorem 1.2]{GrigorCapacity15} and Proposition \ref{prop:CSandCE}. Indeed, it is obvious that the cutoff Sobolev inequality of the present work implies the generalized capacity condition in \cite[Page 1492]{GrigorCapacity15}.
\end{proof}

Lastly, we prove the version of the main theorem in Introduction. Here, we will assume that the ambient metric space is geodesic but we drop the doubling assumption of $\mu$.

\begin{proof}[Proof of Theorem \ref{thm:SUPER}]
    Since the ambient metric space $(X,d)$ is geodesic, we can use the characterization of the heat kernel estimates \ref{eq:HK} of Barlow--Bass--Kumagai \cite[Theorems 1.15-1-16]{barlow2006stability}.
    The reasoning is identical to the one in the proof of Theorem \ref{thm:main}.
\end{proof}

\subsection{Regularity estimates}\label{subsec:Regularity}
We establish some further regularity estimates related to the cutoff energy condition.
First, we prove a two-point estimate for a sharp maximal type function. The method is the same as in \cite[Section 4.3]{MaxFun2021}.

\begin{lemma}\label{lemma:Campanato}
    Let $p \geq 1, \, f \in L^p(X,\mu)$, $R > 0$ and $\delta > 0$.
    Let $M : X \to \R$ be the sharp maximal type function
    \[
        M(x) := \sup_{x \in B(y,r)} r^{-\delta} \kint_{B(y,r)} \abs{f - f_{B(y,r)}}^p \, d\mu
    \]
    where the supremum is taken over all $y \in X$ and $0 < r \leq R$ such that $x \in B(y,r)$.
    Then, there is a constant $C \geq 1$ such that for $\mu$-almost every $x \in X$ the two-point estimate
    \[
        \abs{f(x) - f(y)} \leq C d(x,y)^{\frac{\delta}{p}}(M(x)^{\frac{1}{p}} + M(y)^{\frac{1}{p}})
    \]
    holds for $\mu$-almost every $y \in B(x,R/4)$.
    This result is quantitative in the sense that $C$ depends only on the doubling constant of $\mu$ in \eqref{eq:muDoubling}, $\delta$ and $p$.
\end{lemma}

\begin{proof}
    By the Lebesgue differentiation theorem; see \cite[Chapter 1]{Heinonen}, it is sufficient to verify the two-point estimate when $x$ and $y$ are both points $z \in X$ such that
    \begin{equation}\label{eq:CampanatoLebesgue}
        f(z) = \lim_{r \to 0^+} \kint_{B(z,r)} f\, d\mu.
    \end{equation}
    Let $x \in X$ and $y \in B(x,R/4)$ be such points, and we assume that they are distinct.
    For simplicity, we denote $r:=d(x,y) > 0$. We estimate
    \[
        \abs{f(x) - f(y)} \leq \abs{f(x) - f_{B(x,2r)}} + \abs{f(y) - f_{B(x,2r)}}
    \]
    and consider the first term in the right hand side of the previous display. By using \eqref{eq:CampanatoLebesgue}, it follows from the telescoping argument in the proof of Proposition \ref{prop:LocalUD},
    \begin{align*}
        & \quad \, \abs{f(x) - f_{B(x,2r)}}
        \leq \sum_{i = -1}^\infty \abs{f_{B(x,2^{-i}r)} - f_{B(x,2^{-(i+1)}r)}}\\
        \lesssim & \, \quad \sum_{i = -1}^\infty \kint_{B(x,2^{-i})} \abs{f(z)-f_{B(x,2^{-i})r}}\, d\mu(z)\\
        \lesssim & \, \quad r^{\delta/p} \sum_{i = -1}^\infty 2^{-i/(\delta p)} \left( (2^{-i}r)^{-\delta} \kint_{B(x,2^{-i})} \abs{f(z)-f_{B(x,2^{-i})r}}^p\, d\mu(z)\right)^{\frac{1}{p}}\\
        \leq & \quad r^{\delta/p} M(x)^{\frac{1}{p}}\sum_{i = -1}^\infty 2^{-i/(\delta p)} \lesssim r^{\delta/p} M(x)^{\frac{1}{p}}.
    \end{align*}
    By a similar argument,
    \begin{align*}
        & \abs{f(y) - f_{B(x,2r)}}
        \leq \abs{f(y) - f_{B(y,4r)}} + \abs{f_{B(x,2r)} - f_{B(y,4r)}}\\
        \lesssim \quad & r^{\delta/p}M(y) + r^{\delta/p} \left( (4r)^{-\delta} \kint_{B(y,4r)} \abs{f(z) - f_{B(y,4r)}}^p \, d\mu(z) \right)^{\frac{1}{p}}
        \lesssim \, r^{\delta/p}M(y)^{\frac{1}{p}}.
    \end{align*}
    The two-point estimate now follows by combining the estimates above.
\end{proof}

We now show that the cutoff energy condition implies Hölder reguarity.

\begin{corollary}\label{cor:holder}
    Let $\delta > 0$ and assume that the strongly local regular Dirichlet form $(\mathcal{E},\mathcal{F})$ on $L^2(X,\mu)$ satisfies \ref{eq:PI(Dir)}.
    Let $x \in X$ and $R > 0$ and assume that $\xi \in \mathcal{F}$, not necessarily continuous, satisfies $0 \leq \xi \leq 1$ $\mu$-almost everywhere and the energy upper bounds \ref{eq:CE(Dir)} for the open ball $B(x,R)$.
    Then the precise representative $\tilde{\xi}$ satisfies the global Hölder regularity
    \[
        \abs{\tilde{\xi}(y) - \tilde{\xi}(z)} \leq C\left(\frac{d(y,z)}{R}\right)^{\delta/2} \text{ for all } y,z \in X.
    \]
    This result is quantitative in the sense that $C$ depends only on $\delta$, the constants associated to $\Psi$ in \eqref{eq:Psi}, the constants in \ref{eq:CE(Dir)}, \ref{eq:PI(Dir)} and \eqref{eq:muDoubling}.
\end{corollary}

\begin{proof}
    It follows from the Poincar\'e inequality \ref{eq:PI(Dir)} and the energy upper bound \ref{eq:CE(Dir)} that the following Campanato type semi-norm has the bound
    \[
        \sup_{ \substack{y \in X \\ 0 < r \leq R/\sigma } } r^{-\delta}\kint_{B(y,r)} \abs{\xi - \xi_{B(y,r)}}^p \, d\mu \lesssim R^{-\delta}.
    \]
    The constant $\sigma$ is from \ref{eq:PI(Dir)}.
    By combining this with the two-point estimate in Lemma \ref{lemma:Campanato}, we get that for $\mu$-almost every $y \in X$ and $\mu$-almost every $z \in B(y,R/(4\sigma))$,
    \[
        \abs{\xi(y) - \xi(z)} \lesssim \left(\frac{d(y,z)}{R}\right)^{\delta/2}.
    \]
    By using the fact that $0 \leq \xi \leq 1$ holds $\mu$-almost everywhere, we see that the Hölder estimate in the previous display holds for $\mu$-almost every $y,z \in X$.

    Now, by using the fact the Hölder regularity holds $\mu$-almost everywhere, it follows easily from the definition of the precise representative in \eqref{eq:Exact} that $\tilde{\xi}$ satisfies the desired Hölder regularity everywhere.
\end{proof}

We prove that the cutoff Sobolev inequality self-improves, quantitatively. This is very similar to \cite[Theorem 5.4]{barlow2004stability}.

\begin{corollary}
    Let $\delta > 0$ and assume that the strongly local regular Dirichlet form $(\mathcal{E},\mathcal{F})$ on $L^2(X,\mu)$ satisfies \ref{eq:PI(Dir)}.
    Let $x \in X$, $R \in (0,\diam(X))$ and assume that
    $\xi \in \mathcal{F}$ is a cutoff function for $B(x,r) \subseteq B(x,2r)$ satisfying the condition in \ref{eq:CS(Dir)}. Then there are constants $C,\sigma \geq 1$ and $q > 2$ such that the Sobolev--Poincar\'e inequality
    \[
        \left( \int_{B(y,r)} \abs{f_{qc} - f_{B(y,r)}}^q \, d\Gamma\Span{\xi} \right)^{\frac{1}{q}} \leq C \left( \frac{r}{R} \right)^{\frac{\delta}{q}} \left(\frac{\mu(B(y,r))}{\Psi(r)}\right)^{\frac{1}{q} - \frac{1}{2}}\left( \int_{B(y,\sigma r)} \, d\Gamma\Span{f} \right)^{\frac{1}{2}}.
    \]
    holds for all $y \in X$, $0 < r \leq 2R$ and $f \in \mathcal{F}$. This result is quantitative in the sense that $q$ depends only on the lower exponent $Q_L$ of $\mu$ in \eqref{eq:RevDoubling}, the the upper exponent $\beta_U$ of $\Psi$ in \eqref{eq:Psi} and $\delta$. The constants $C,\sigma$ depend only on the constants in \ref{eq:PI}, \ref{eq:CS(Dir)} and the constants associated to $\Psi$ in \eqref{eq:Psi}.
\end{corollary}

\begin{proof}
    Fix $q > 2$.
    By the same argument as in the proof of Proposition \ref{prop:CSandCE}, it is sufficient to prove the inequality in the claim for continuous $f$. We also have for all $y \in X$ and $0 < r \leq 2R$,
    \[
    \Gamma\Span{\xi}(B(y,r)) = \int_{B(y,r)} \, d\Gamma\Span{\xi} \lesssim \left( \frac{r}{R} \right)^{\delta}\frac{\mu(B(y,r))}{\Psi(r)}.
    \]

    Now, consider $y,z \in X$ and $r,s > 0$ such that $0 < s \leq r$ and $B(y,r) \cap B(z,s) \neq \emptyset$.
    By using the doubling properties of $\mu$ and $\Psi$, we get
    \begin{align*}
        \Gamma\Span{\xi}(B(z,s)) & \lesssim \left(\frac{s}{R}\right)^\delta \frac{\mu(B(z,s))}{\Psi(s)}\\
        & = \left(\frac{s}{R}\right)^\delta \left(\frac{\mu(B(z,s))}{\Psi(s)}\right)^{1 - \frac{q}{2}} \left(\frac{\mu(B(z,s))}{\Psi(s)}\right)^{\frac{q}{2}}\\
        & \lesssim L_0^q \left(\frac{s}{r}\right)^{\delta + \left(1 - \frac{q}{2}\right)(Q_L - \beta_U)}\left(\frac{\mu(B(z,s))}{\Psi(s)}\right)^{\frac{q}{2}},
    \end{align*}
    where $Q_L$ and $\beta_U$ are as in the claim and $L_0$ is given by 
    \[
    L_0 := C\left(\frac{\mu(B(y,r))}{\Psi(r)}\right)^{\frac{1}{q} - \frac{1}{2}} \left(\frac{r}{R}\right)^{\frac{\delta}{q}}. 
    \]
    By choosing $q > 2$ such that $\delta + (1 - q/2)(Q_L - \beta_L) > 0$, which is possible by $\delta > 0$,
    the Sobolev--Poincar\'e inequality in the claim follows from Proposition \ref{prop:LocalUD}.
\end{proof}

\subsection{Concluding remarks}\label{subsec:concludingremarks}
We provide some further discussion on the results of the work.

\begin{remark}\label{rem:p-CS}
    The techniques developed in Sections
    \ref{sec:SPI} and \ref{sec:Dir} are quite general and do not rely on the linearity of Dirichlet forms. Consequently, Theorem \ref{thm:SobolevPoincare} and the proofs in Section \ref{sec:Dir} can be applied to obtain analogous results in non-linear settings.
    Indeed, consider a $p$-energy structure $(\mathcal{F}_p,\Gamma_p,\Psi)$ within a suitable framework of non-linear potential theory in metric spaces; see for instance \cite{KajinoContraction,Sasaya2025}.
    Then, the analogous counterpart of the cutoff energy condition
    \[
        \int_{B(y,r)} \,d\Gamma_p\Span{\xi} \leq C\left( \frac{r}{R} \right)^\delta \frac{\mu(B(y,r))}{\Psi(r)}
    \]
    is equivalent to the \emph{$p$-cutoff Sobolev inequality}
    \begin{equation*}
    \int_{B(y,r)} \abs{f_{qc}}^p \, d\Gamma_p\Span{\xi} \leq C \left(\frac{r}{R}\right)^\delta \left( \int_{B(y,2r)} \, d\Gamma_p\Span{f} + \frac{1}{\Psi(r)} \int_{B(y,2r)} \abs{f}^p \, d\mu \right).
    \end{equation*}
    See also \cite{yang2025energy,yang2025singularity} for recent studies on the simplified variant of the $p$-cutoff Sobolev inequality.
\end{remark}

In \cite{yang2025singularity}, Yang observed that the cutoff Sobolev inequality follows from a suitable Morrey type inequality. In the remark below, we show that Proposition \ref{prop:CSandCE} provides a simpler argument to this end.
However, we also note that, under Morrey's inequality, the sub-Gaussian heat kernel estimates can be reached without the cutoff Sobolev inequality \cite{barlow2005characterization}.

\begin{remark}
    Assume that $(X,d,\mu)$ is $Q$-Ahlfors regular, meaning $\mu(B(x,r)) \approx r^Q$ for all $x \in X$ and $r \in (0,2\diam(X)]$, and that $\Psi(r) = r^\beta$ for $\beta > Q$.
    Assume also that the strongly local regular Dirichlet form $(\mathcal{E},\mathcal{F})$ on $L^2(X,\mu)$ satisfies
    \hyperref[eq:PI(Dir)]{$\textup{PI}(\beta)$}
    and \hyperref[eq:Ucap(Dir)]{$\textup{Cap}_\leq(\beta)$}.
    Let $\varphi \in \mathcal{F}$ be a cutoff function
    for $B(x,r) \subseteq B(x,2r)$ provided by 
    \hyperref[eq:Ucap(Dir)]{$\textup{Cap}_\leq(\beta)$}.
    Now, by performing the very rough estimate
    \[
        \int_{B(y,r)} \, d\Gamma\Span{\xi} \leq \mathcal{E}(\varphi) \lesssim R^{Q-\beta} = \left(\frac{r}{R}\right)^{\beta-Q} r^{Q-\beta} \approx \left(\frac{r}{R}\right)^{\beta-Q} \frac{\mu(B(x,r))}{r^\beta},
    \]
    we obtain the cutoff energy condition \hyperref[eq:CE(Dir)]{$\textup{CE}_\delta(\beta)$} for $\delta := \beta - Q > 0$, we have verified the cutoff energy condition \hyperref[eq:CE(Dir)]{$\textup{CE}_\delta(\beta)$}.
    This observation along with Theorem \ref{thm:main} provides a short argument that the dimension condition $\beta > Q$ along with \hyperref[eq:PI(Dir)]{$\textup{PI}(\beta)$}
    and \hyperref[eq:Ucap(Dir)]{$\textup{Cap}_\leq(\beta)$} implies the heat kernel estimates \hyperref[eq:HKMain]{$\textup{HKE}(\beta)$}.
\end{remark}

We briefly discuss the resistance conjecture of Grigor’yan, Hu, and Lau, and see \cite[Remark 1.2]{murugan2023note} and \cite[Section 6.3]{murugan2024heat} for further discussions.

\begin{conjecture}[Conjecture 4.15 \cite{grigor2014heat}]
    The strongly local regular Dirichlet form $(\mathcal{E},\mathcal{F})$ on $L^2(X,\mu)$ satisfies the heat kernel estimates \ref{eq:HKMain} if and only  $(\mathcal{E},\mathcal{F})$ satisfies both the Poincar\'e inequality \ref{eq:PI(Dir)} and the upper capacity estimate \ref{eq:Ucap(Dir)}.
\end{conjecture}

In other words, the resistance conjecture asserts that \ref{eq:CE(Dir)} in the statement of Theorem \ref{thm:main} can be replaced by \ref{eq:Ucap(Dir)}. By the strong locality, it follows easily that \ref{eq:CE(Dir)} implies \ref{eq:Ucap(Dir)}.
However, as the remark below shows, the former is, a priori, a stronger requirement than the latter.

\begin{remark}
    It is a notable fact that the cutoff energy condition with $\delta = 0$, namely \hyperref[eq:CE(Dir)]{$\textup{CE}_0(\Psi)$}, for potentially discontinuous $\xi$ follows from \ref{eq:Ucap(Dir)} and the log-Caccioppoli inequality \cite[Lemma 7.1]{GrigorCapacity15}.
    As explained in Remark \ref{rem:delta=0}, \hyperref[eq:CE(Dir)]{$\textup{CE}_0(\Psi)$} is not sufficiently strong to derive the cutoff Sobolev inequality \hyperref[eq:CS(Dir)]{$\textup{CE}_0(\Psi)$} with our method.
    Also, according to Corollary \ref{cor:holder}, \ref{eq:CE(Dir)} for $\delta > 0$ implies Hölder continuity, while the estimate for $\delta = 0$ does not provide such a priori regularity. For instance, given any $n \geq 2$ and $f \in W^{1,n}(\R^n)$, it follows from Hölder's inequality,
\[
    \int_{B(y,r)} \abs{\nabla f}^2 \,dx \leq C(n) \left( \int_{\R^n} \abs{\nabla f}^n \, dx\right)^{n/2} r^{n-2},
\]
which is \hyperref[eq:CE(Dir)]{$\textup{CE}_0(\Psi)$}. 
It is an elementary result in the Sobolev space theory that 
$W^{1,n}(\R^n)$ contains discontinuous functions. In particular, there are functions that satisfy \hyperref[eq:CE(Dir)]{$\textup{CE}_\delta(\Psi)$} for $\delta = 0$ but not for any $\delta > 0$.
In any case, according to \cite[Theorem 1.2]{GrigorCapacity15}, if one could justify the implication \hyperref[eq:CE(Dir)]{$\textup{CE}_0(\Psi)$} $\Rightarrow$ \hyperref[eq:CS(Dir)]{$\textup{CS}_0(\Psi)$}, this would positively resolve the resistance conjecture.
\end{remark}

\begin{remark}
    The proof of Proposition \ref{prop:BBK} has a simpler proof if we use the fact that \ref{eq:HKMain} implies the \emph{simplified cutoff Sobolev inequality} introduced in \cite{andres2015energy}.
    Indeed, the proof of Proposition \ref{prop:BBK} shows that the cutoff function $\xi$ is a subharmonic in suitably chosen open balls $B(x,2r)$. Combining this fact with 
    the reverse Poincar\'e inequality \cite[Lemma 3.3]{kajino2020singularity} and 
    the Hölder regularity of $\xi$ provided by Lemma \ref{lemma:h}, 
    \[
    \int_{B(x,r)} \, d\Gamma\Span{\xi} \lesssim \frac{1}{\Psi(r)} \int_{B(x,2r)} (f - f_{B(x,2r)})^2 \,d\mu \lesssim \left( \frac{r}{R_0} \right)^\delta \left(\frac{\mu(B(x,r))}{\Psi(r)}\right).
    \] 
    Thus, Hölder continuous subharmonic functions appear to be natural candidates for the functions in the cutoff energy condition. This observation perhaps could be used to further analyze the resistance conjecture.
\end{remark}

\section{Examples of Poincar\'e inequalities}\label{Sec:Applications}

The purpose of the final section of the work is to discuss some simple applications of Theorem \ref{thm:SobolevPoincare}.
We note that some of these are already available in the literature and the primary motivation is to provide some helpful examples for the reader.

Throughout the section we use the same convention and notation as discussed in the beginning of Section \ref{sec:SPI}, and consider a fixed $p$-energy structure $(\mathcal{F}_p,\Gamma_p,\Psi)$.
We denote the lower and upper exponents of $\mu$ and $\Psi$ by $Q_L$ $Q_U$, and $\beta_L$ and $\beta_U$, respectively. The notation $\tilde{f}$ always refers to the precise representative in \eqref{eq:Exact}.
\subsection{Classical theorems}

Recall that the classical Sobolev--Poincar\'e inequality reads as follows.
Let $n \geq 2$ be an integer, $p \in [1,n)$ and $p^* := pn/(p-n)$.
Then for all $x \in \R^n$, $r > 0$ and $f \in W^{1,p}(\R^n)$,
\begin{equation*}
\left(\kint_{B(x,r)} \abs{f - f_{B(x,r)}}^{p^*} \, dx\right)^{\frac{1}{p^*}} \leq C(n,p) r \left(\kint_{B(x,r)} \abs{\nabla f}^p dx \right)^{\frac{1}{p}}.
\end{equation*}
The following proposition recovers the classical Sobolev--Poincar\'e inequality, except for the endpoint $q = p^*$. The endpoint estimate in many settings can be obtained using the truncation method of Maz'ya \cite{Mazya85}; see for instance \cite{bakry1995sobolev,BJORNKala,SobolevMeetsPoincare,KinnunenKorte08}.

\begin{proposition}\label{prop:CSPI}
    Let $p^* \in (p,\infty]$ be given by
    \[
    p^* := 
    \begin{cases}
        pQ_U/(Q_U - \beta_L) & \text{ if } \beta_L < Q_U\\
        \infty & \text{ if } \beta_L \geq Q_U.
    \end{cases}
    \]
    Then the reference measure $\mu$ satisfies the following Sobolev--Poincar\'e inequality for all $q \in [1,p^*)$.
    There are constant $C,\sigma \geq 1$ such that for all $f \in \mathcal{F}_p$,
    \[
        \left( \kint_{B(x,r)} \abs{f - f_{B(x,r)}}^q \, d\mu \right)^{\frac{1}{q}} \leq C \left(\frac{\Psi(x,r)}{\mu(B(x,r))} \int_{B(x,\sigma r)} \, d\Gamma_p\Span{f} \right)^{\frac{1}{p}}.
    \]
\end{proposition}

\begin{remark}
    We do not need to use the precise representative here because $\tilde{f} = f$ $\mu$-almost everywhere.
\end{remark}

\begin{proof}[Proof of Proposition \ref{prop:CSPI}]
    The case $q = p$ is just the Poincar\'e inequality \ref{eq:PI}, and the cases $q \in [1,p)$ follow from Hölder's inequality,
    \begin{align*}
        \left(\kint_{B(x,r)} \abs{f - f_{B(x,r)}}^q \, d\mu\right)^{\frac{1}{q}} & \leq \left(\frac{1}{\mu(B(x,r))} \int_B \abs{f - f_{B(x,r)}}^p d\mu \right)^{\frac{1}{p}}\\
        & \lesssim \left(\frac{\Psi(x,r)}{\mu(B(x,r))} \int_{B(x,\sigma r)} \, d\Gamma_p\Span{f} \right)^{\frac{1}{p}}.
    \end{align*}
    
    Next, we consider the case $q > p$. Let
    \[
        \Theta(x,r) := \mu(B(x,r))^{\frac{1}{q}-\frac{1}{p}}\Psi(x,r)^{\frac{1}{p}}.
    \]
    Observe that \ref{T2} in Theorem \ref{thm:SobolevPoincare} for $\nu = \mu$ and $\Theta$ given in the previous display is equivalent to the objective of the proof.
    Also note that
    \[
        \mu(B(x,r))^{\frac{1}{q}} = \Theta(B(x,r)) \left( \frac{\Psi(x,r)}{\mu(B(x,r))} \right)^{\frac{1}{p}} .
    \]
    Thus, we are done once we have verified that $\Theta$ satisfies \eqref{eq:Psi} for $q \in (p,p^*)$.
    
    Let $x,y \in X$ and $0 < r \leq R$ so that $B(x,R) \cap B(y,r) \neq \emptyset$.
    By the doubling properties of $\mu$ and $\Psi$, 
    \[
        \left(\frac{R}{r}\right)^{Q_U + \frac{q}{p}(\beta_L - Q_U)} \lesssim \left(\frac{\Theta(x,R)}{\Theta(y,r)}\right)^q \lesssim \left(\frac{R}{r}\right)^{Q_L + \frac{q}{p}(\beta_U - Q_L)}.
    \]
    If $\beta_L < Q_U$, then the exponents in the display are positive if and only if $q \in [1,p^*)$.
    On the other hand, if $\beta_L \geq Q_U$, then the exponents are positive for all $q \geq 1$.
\end{proof}

Recall that the classical Morrey's inequality states that, when $n \geq 2$ is an integer and $p > n$, the Sobolev functions $f \in W^{1,p}(\R^n)$ satisfy the local Hölder regularity
\[
    \abs{f(x) - f(y)} \leq C(n,p)\abs{x - y}^{1 - \frac{n}{p}}\left( \int_{B(x,2\abs{x-y})} \abs{\nabla f}^p \, dx \right)^{\frac{1}{p}}.
\]
Analogous estimates in metric spaces have been obtained in \cite[Theorem 5.1]{SobolevMetPoincare}. See also \cite[Theorem 3.21]{kigami}, \cite[Theorem 1.3]{barlow2005characterization}.

\begin{proposition}\label{prop:Morrey}
    If $\beta_L > Q_U$, then the following Morrey type inequality holds.
    There are constants $C,\sigma \geq 1$ such that for all $z \in X$ and $r_0 > 0$,
    \[
        \sup_{x \in B(z,r_0)} \abs{\tilde{f}(x) - f_{B(z,r_0)}} \leq C \left(\frac{\Psi(x,r_0)}{\mu(B(x,r_0))} \int_{B(x,\sigma r_0)} \, d\Gamma_p\Span{f}\right)^{\frac{1}{p}}.
    \]
\end{proposition}

\begin{proof}
    Fix $z \in X$ and $r_0 > 0$.
    The key idea is to apply Theorem \ref{thm:SobolevPoincare} for the Dirac delta measures. We define
    \[
    \Theta(x,r) := \left(\frac{\Psi(x,r)}{\mu(B(x,r))}\right)^{\frac{1}{p}} .
    \]
    It follows from the doubling properties of $\mu$ and $\Psi$ that for all $x,y \in X$ and $0 < r \leq R$,
    \[
    \left(\frac{R}{r}\right)^{\beta_L - Q_U} \lesssim \frac{\Theta(x,R)}{\Theta(y,r)}
    \lesssim \left(\frac{R}{r}\right)^{\beta_U - Q_L}.
    \]
    Hence, $\Theta$ satisfies the doubling condition in \eqref{eq:Psi} according to $\beta_L > Q_U$.

    Now, fix an arbitrary point $x \in B(z,r_0)$ and consider the Dirac delta measure $\delta_x$ concentrated at $x$.
    Since we have the obvious inequality
    \[
        \delta_x(B(y,r)) \leq \Theta(y,r)^p \frac{\mu(B(y,r))}{\Psi(y,r)},
    \]
    we have by Theorem \ref{thm:SobolevPoincare},
    \[
        \abs{\tilde{f}(x) - f_{B(z,r_0)}}^p = \int_{B(z,r_0)} \abs{\tilde{f} - f_{B(z,r_0)}}^p \, d\delta_x \lesssim \frac{ \Psi(x,r_0) }{\mu(B(z,r_0))}\int_{B(z,\sigma r_0)} \, d\Gamma_p\Span{f}.
    \]
    Because the point $x \in B(z,r_0)$ was arbitrary, we may replace the left-hand side of the previous inequality by the supremum over $x \in B(z,r_0)$.
\end{proof}

\subsection{Balance condition}\label{subsec:Balance}
Chanillo and Wheeden introduced the balance condition of Muckenhoupt weights \cite{chanillowheeden85} and studied its relation to Sobolev--Poincar\'e type inequalities.
The analogous notion in the present language is the following.
We say that a doubling measure $\nu$ satisfies the \emph{$(q,p)$-balance condition} if for all $x,y \in X$ and $0 < r \leq R$ such that $B(x,R) \cap B(y,r) \neq \emptyset$,
\begin{equation*}
        \left(\frac{\Psi(y,r)}{\Psi(x,R)}\right)^{\frac{1}{p}} \left(\frac{ \nu(B(y,r)) }{ \nu(B(x,R)) }  \right)^{\frac{1}{q}} \leq K\left( \frac{\mu(B(y,r))}{\mu(B(x,R))} \right)^{\frac{1}{p}}.
\end{equation*}
We rewrite the balance condition to resemble the condition \ref{T1} by setting
\[
    \Theta(x,r) := \nu(B(x,r))^{\frac{1}{q}} \left(\frac{\Psi(x,r)}{\mu(B(x,r))}\right)^{\frac{1}{p}},
\]
or equivalently
\[
    \nu(B(x,r))^{\frac{1}{q}} \leq \Theta(x,r)\left( \frac{\mu(B(x,r))}{\Psi(x,r)} \right)^{\frac{1}{p}}.
\]
Hence, \ref{T1} in Theorem \ref{thm:SobolevPoincare}.
Now, the balance condition is equivalent to stating that $\Theta(y,r) \leq K \Theta(x,R)$ whenever $x,y \in X$ and $0 < r \leq R$ such that $B(x,R) \cap B(y,r) \neq \emptyset$. This is not quite strong enough for us to apply Theorem \ref{thm:SobolevPoincare}. 
We note that Chanillo--Wheeden avoids this issue.
Nevertheless, we observe that the $(q,p)$-balance condition and the doubling property of $\nu$ implies the following variant of the balance condition for all $t \in [1,q)$,
\[
    \left(\frac{\Psi(y,r)}{\Psi(x,R)}\right)^{\frac{1}{p}} \left(\frac{ \nu(B(y,r)) }{ \nu(B(x,R)) }  \right)^{\frac{1}{t}} \leq K' \left(\frac{r}{R} \right)^{\delta} \left( \frac{\mu(B(y,r))}{\mu(B(x,R))} \right)^{\frac{1}{p}}.
\]
A study by Kinnunen, Korte, Lehrbäck and Vähäkangas \cite{kinnunenvahakangas2019} shows that the above variant of the balance condition, which they termed the \emph{bumbed balance condition}, is related to a Keith–Zhong type self-improvement \cite{keithzhong} in the two-measure settings of analysis on metric spaces.
A quite similar condition was studied in the Dirichlet form setting in Barlow--Murugan \cite[Definition 4.1]{barlow2018stability} and later in Barlow--Chen--Murugan \cite[Definition 6.2]{barlow2020stability}.

\begin{proposition}\label{prop:balance}
    Let $q > p$ and assume that $\nu$ is a doubling measure on $X$ that satisfies the $(q,p)$-balance condition. Then $\nu$ satisfies the following Sobolev--Poincar\'e inequality for all $t \in [1,q)$. There are constants $C,\sigma \geq 1$ so that for all open balls $x \in X$, $r > 0$ and $f \in \mathcal{F}_p$,
    \[
        \left( \kint_{B(x,r)} \abs{\tilde{f} - f_{B(x,r)}}^t \, d\nu\right)^{\frac{1}{t}} \leq C \left( \frac{\Psi(x,r)}{\mu(B(x,r))}\int_{B(x,\sigma r)} \, d\Gamma_p\Span{f}  \right)^{\frac{1}{p}}.
    \]
\end{proposition}

\begin{proof}
    As in the proof of Proposition \ref{prop:CSPI}, we only consider $t \in [p,q)$ since the rest would follow from Hölder's inequality.

    Fix $t \in [p,q)$ and let $\alpha_L > 0$ be a lower exponent of $\nu$ in \eqref{eq:RevDoubling}.
    We use the balance condition and the doubling property \eqref{eq:RevDoubling} of $\nu$ to get
    \begin{align*}
        \left(\frac{\Psi(y,r)}{\Psi(x,R)}\right)^{\frac{1}{p}} \left(\frac{ \nu(B(y,r)) }{ \nu(B(x,R)) }  \right)^{\frac{1}{t}} & =  \left(\frac{ \nu(B(y,r)) }{ \nu(B(x,R)) }  \right)^{\frac{1}{t} - \frac{1}{q}}\left(\frac{\Psi(y,r)}{\Psi(x,R)}\right)^{\frac{1}{p}} \left(\frac{ \nu(B(y,r)) }{ \nu(B(x,R))) }  \right)^{\frac{1}{q}}\\
        & \lesssim  \left(\frac{ \nu(B(y,r)) }{ \nu(B(x,R)) }  \right)^{\frac{1}{t} - \frac{1}{q}}
        \left( \frac{\mu(B(y,r))}{\mu(B(x,R))} \right)^{\frac{1}{p}}\\
        & \lesssim \left( \frac{r}{R} \right)^{\alpha_L\left(\frac{1}{t} - \frac{1}{q} \right)}\left( \frac{\mu(B(x,r))}{\mu(B(x,R))} \right)^{\frac{1}{p}}
    \end{align*}
    for all $x,y \in X$ and $0 < r \leq R$.
    Thus, it follows from the previous estimates that
    \[
        \Theta(x,r) := \nu(B(x,r))^{\frac{1}{t}} \left( \frac{\Psi(x,r)}{\mu(B(x,r))} \right)^{\frac{1}{p}}
    \]
    is a scale function with lower exponent $\delta := (1/t - 1/q)\alpha_L > 0$.
    Note that the other inequality in \eqref{eq:Psi} follows by combining the doubling properties of $\nu,\mu$ and $\Psi$ and the fact that $\delta > 0$.
    Theorem \ref{thm:SobolevPoincare} now implies the desired Sobolev--Poincar\'e inequality.
\end{proof}

We also discuss a simple corollary of Proposition \ref{prop:balance}. We consider the analytic condition $\Psi(x,r) = \mu(B(x,r))$, which in the usual $W^{1,p}(\R^n)$ setting is understood as the dimension condition $p = n$. 

\begin{corollary}\label{cor:ConfInvariance}
    Let $q \geq 1$, $\nu$ be a doubling measure on $X$ and assume that $\Psi(x,r) = \mu(B(x,r))$ for all $x \in X$ and $r > 0$.
    Then $\nu$ satisfies the following Sobolev--Poincar\'e inequality. There are constants $C,\sigma \geq 1$ such that for all $x \in X$, $r > 0$ and $f \in \mathcal{F}_p$,
    \[
        \left( \kint_{B(x,r)} \abs{\tilde{f} - f_{B(x,r)}}^q \, d\nu \right)^{\frac{1}{q}}
        \leq C \left( \int_{B(x,\sigma r)} \, d\Gamma_p\Span{f} \right)^{\frac{1}{p}}.
    \]
\end{corollary}

\begin{proof}
    By the condition $\mu(B(x,r)) = \Psi(x,r)$, the measure $\nu$ satisfies the $(q,p)$-balance condition for all $q \geq 1$. The claim now follows from Proposition \ref{prop:balance}.
\end{proof}

As an application of Corollary \ref{cor:ConfInvariance}, we obtain a more general version of a certain pair of two measure Poincar\'e inequalities in Murugan--Shimizu \cite[Proposition 9.21]{murugan2023first}.
The condition $\Psi(x,r) = \mu(B(x,r))$ in their framework holds exactly when $p$ is equal to the Ahlfors regular conformal dimension of the Sierpi\'nski carpet \cite[Assumption 9.16]{murugan2023first}.

\bibliographystyle{acm}
\bibliography{PI}

\begin{thebibliography}{10}

\bibitem{adams1973trace}
{\sc Adams, D.~R.}
\newblock A trace inequality for generalized potentials.
\newblock {\em Studia Math. 48\/} (1973), 99--105.

\bibitem{andres2015energy}
{\sc Andres, S., and Barlow, M.~T.}
\newblock Energy inequalities for cutoff functions and some applications.
\newblock {\em J. Reine Angew. Math. 699\/} (2015), 183--215.

\bibitem{AnttilaReflectedDiff}
{\sc Anttila, R.}
\newblock Sub-{G}aussian heat kernel estimates for reflected diffusion on inner uniform domains.
\newblock {\em arXiv preprint arXiv:2510.04725\/} (2025).

\bibitem{anttila2024constructions}
{\sc Anttila, R., and Eriksson-Bique, S.}
\newblock On constructions of fractal spaces using replacement and the combinatorial loewner property.
\newblock {\em arXiv preprint arXiv:2406.08062\/} (2024).

\bibitem{AEBSLaaksoSobolev2025}
{\sc Anttila, R., Eriksson-Bique, S., and Shimizu, R.}
\newblock Construction of self-similar energy forms and singularity of {S}obolev spaces on {L}aakso-type fractal spaces.
\newblock {\em arXiv preprint arXiv:2503.13258\/} (2025).

\bibitem{bakry1995sobolev}
{\sc Bakry, D., Coulhon, T., Ledoux, M., and Saloff-Coste, L.}
\newblock Sobolev inequalities in disguise.
\newblock {\em Indiana Univ. Math. J. 44}, 4 (1995), 1033--1074.

\bibitem{barlow}
{\sc Barlow, M.~T.}
\newblock Diffusions on fractals.
\newblock In {\em Lectures on probability theory and statistics ({S}aint-{F}lour, 1995)}, vol.~1690 of {\em Lecture Notes in Math.} Springer, Berlin, 1998, pp.~1--121.

\bibitem{barlow2013analysis}
{\sc Barlow, M.~T.}
\newblock Analysis on the {S}ierpi{\'n}ski carpet.
\newblock In {\em Analysis and geometry of metric measure spaces}, vol.~56 of {\em CRM Proc. Lecture Notes}. Amer. Math. Soc., Providence, RI, 2013, pp.~27--53.

\bibitem{barlow2017random}
{\sc Barlow, M.~T.}
\newblock {\em Random walks and heat kernels on graphs}, vol.~438 of {\em London Mathematical Society Lecture Note Series}.
\newblock Cambridge University Press, Cambridge, 2017.

\bibitem{barlow1989construction}
{\sc Barlow, M.~T., and Bass, R.~F.}
\newblock The construction of {B}rownian motion on the {S}ierpi{\'n}ski carpet.
\newblock {\em Ann. Inst. H. Poincar\'e{} Probab. Statist. 25}, 3 (1989), 225--257.

\bibitem{barlow1999brownian}
{\sc Barlow, M.~T., and Bass, R.~F.}
\newblock Brownian motion and harmonic analysis on {S}ierpi{\'n}ski carpets.
\newblock {\em Canad. J. Math. 51}, 4 (1999), 673--744.

\bibitem{barlow2004stability}
{\sc Barlow, M.~T., and Bass, R.~F.}
\newblock Stability of parabolic {H}arnack inequalities.
\newblock {\em Trans. Amer. Math. Soc. 356}, 4 (2004), 1501--1533.

\bibitem{barlow2006stability}
{\sc Barlow, M.~T., Bass, R.~F., and Kumagai, T.}
\newblock Stability of parabolic {H}arnack inequalities on metric measure spaces.
\newblock {\em J. Math. Soc. Japan (correction in arXiv:2001.06714) 58}, 2 (2006), 485--519.

\bibitem{barlow2020stability}
{\sc Barlow, M.~T., Chen, Z.-Q., and Murugan, M.}
\newblock Stability of {EHI} and regularity of {MMD} spaces.
\newblock {\em arXiv preprint arXiv:2008.05152\/} (2020).

\bibitem{barlow2005characterization}
{\sc Barlow, M.~T., Coulhon, T., and Kumagai, T.}
\newblock Characterization of sub-{G}aussian heat kernel estimates on strongly recurrent graphs.
\newblock {\em Comm. Pure Appl. Math. 58}, 12 (2005), 1642--1677.

\bibitem{barlow2012equivalence}
{\sc Barlow, M.~T., Grigor'yan, A., and Kumagai, T.}
\newblock On the equivalence of parabolic {H}arnack inequalities and heat kernel estimates.
\newblock {\em J. Math. Soc. Japan 64}, 4 (2012), 1091--1146.

\bibitem{barlow2018stability}
{\sc Barlow, M.~T., and Murugan, M.}
\newblock Stability of the elliptic {H}arnack inequality.
\newblock {\em Ann. of Math. (2) 187}, 3 (2018), 777--823.

\bibitem{bjorn2011nonlinear}
{\sc Bj\"orn, A., and Bj\"orn, J.}
\newblock {\em Nonlinear potential theory on metric spaces}, vol.~17 of {\em EMS Tracts in Mathematics}.
\newblock European Mathematical Society (EMS), Z\"urich, 2011.

\bibitem{BBCapacity23}
{\sc Bj\"orn, A., and Bj\"orn, J.}
\newblock Sharp {B}esov capacity estimates for annuli in metric spaces with doubling measures.
\newblock {\em Math. Z. 305}, 3 (2023), Paper No. 41, 26.

\bibitem{BBLJUHA}
{\sc Bj\"orn, A., Bj\"orn, J., and Lehrb\"ack, J.}
\newblock Volume growth, capacity estimates, {$p$}-parabolicity and sharp integrability properties of {$p$}-harmonic {G}reen functions.
\newblock {\em J. Anal. Math. 150}, 1 (2023), 159--214.

\bibitem{BJORNKala}
{\sc Bj{\"o}rn, J., and Ka{\l}amajska, A.}
\newblock Poincar{\'e} inequalities and compact embeddings from {S}obolev type spaces into weighted {$L^q$} spaces on metric spaces.
\newblock {\em J. Funct. Anal. 282}, 11 (2022), Paper No. 109421, 47.

\bibitem{cao2022p}
{\sc Cao, S., Gu, Q., and Qiu, H.}
\newblock {$p$}-energies on p.c.f. self-similar sets.
\newblock {\em Adv. Math. 405\/} (2022), Paper No. 108517, 58.

\bibitem{chanillowheeden85}
{\sc Chanillo, S., and Wheeden, R.~L.}
\newblock Weighted {P}oincar\'e{} and {S}obolev inequalities and estimates for weighted {P}eano maximal functions.
\newblock {\em Amer. J. Math. 107}, 5 (1985), 1191--1226.

\bibitem{chen2012symmetric}
{\sc Chen, Z.-Q., and Fukushima, M.}
\newblock {\em Symmetric {M}arkov processes, time change, and boundary theory}, vol.~35 of {\em London Mathematical Society Monographs Series}.
\newblock Princeton University Press, Princeton, NJ, 2012.

\bibitem{Perez98}
{\sc Franchi, B., P\'erez, C., and Wheeden, R.~L.}
\newblock Self-improving properties of {J}ohn-{N}irenberg and {P}oincar{\'e} inequalities on spaces of homogeneous type.
\newblock {\em J. Funct. Anal. 153}, 1 (1998), 108--146.

\bibitem{FOT}
{\sc Fukushima, M., \=Oshima, Y.~o., and Takeda, M.}
\newblock {\em Dirichlet forms and symmetric {M}arkov processes}, vol.~19 of {\em De Gruyter Studies in Mathematics}.
\newblock Walter de Gruyter \& Co., Berlin, 1994.

\bibitem{GilbargTrudinger}
{\sc Gilbarg, D., and Trudinger, N.~S.}
\newblock {\em Elliptic partial differential equations of second order}.
\newblock Classics in Mathematics. Springer-Verlag, Berlin, 2001.
\newblock Reprint of the 1998 edition.

\bibitem{grigoryan2009heat}
{\sc Grigor'yan, A.}
\newblock {\em Heat kernel and analysis on manifolds}, vol.~47 of {\em AMS/IP Studies in Advanced Mathematics}.
\newblock American Mathematical Society, Providence, RI; International Press, Boston, MA, 2009.

\bibitem{grigor2014heat}
{\sc Grigor'yan, A., Hu, J., and Lau, K.-S.}
\newblock Heat kernels on metric measure spaces.
\newblock In {\em Geometry and analysis of fractals}, vol.~88 of {\em Springer Proc. Math. Stat.} Springer, Heidelberg, 2014, pp.~147--207.

\bibitem{GrigorCapacity15}
{\sc Grigor'yan, A., Hu, J., and Lau, K.-S.}
\newblock Generalized capacity, {H}arnack inequality and heat kernels of {D}irichlet forms on metric measure spaces.
\newblock {\em J. Math. Soc. Japan 67}, 4 (2015), 1485--1549.

\bibitem{grigor2012two}
{\sc Grigor'yan, A., and Telcs, A.}
\newblock Two-sided estimates of heat kernels on metric measure spaces.
\newblock {\em Ann. Probab. 40}, 3 (2012), 1212--1284.

\bibitem{hajlasz1996sobolev}
{\sc Haj\l~asz, P.}
\newblock Sobolev spaces on an arbitrary metric space.
\newblock {\em Potential Anal. 5}, 4 (1996), 403--415.

\bibitem{SobolevMeetsPoincare}
{\sc Haj\l~asz, P., and Koskela, P.}
\newblock Sobolev meets {P}oincar\'e.
\newblock {\em C. R. Acad. Sci. Paris S\'er. I Math. 320}, 10 (1995), 1211--1215.

\bibitem{SobolevMetPoincare}
{\sc Haj\l~asz, P., and Koskela, P.}
\newblock Sobolev met {P}oincar\'e.
\newblock {\em Mem. Amer. Math. Soc. 145}, 688 (2000).

\bibitem{Heinonen}
{\sc Heinonen, J.}
\newblock {\em Lectures on analysis on metric spaces}.
\newblock Universitext. Springer-Verlag, New York, 2001.

\bibitem{HK}
{\sc Heinonen, J., and Koskela, P.}
\newblock Quasiconformal maps in metric spaces with controlled geometry.
\newblock {\em Acta Math. 181}, 1 (1998), 1--61.

\bibitem{HKST}
{\sc Heinonen, J., Koskela, P., Shanmugalingam, N., and Tyson, J.~T.}
\newblock {\em Sobolev spaces on metric measure spaces}, vol.~27 of {\em New Mathematical Monographs}.
\newblock Cambridge University Press, Cambridge, 2015.
\newblock An approach based on upper gradients.

\bibitem{herman2004p}
{\sc Herman, P.~E., Peirone, R., and Strichartz, R.~S.}
\newblock {$p$}-energy and {$p$}-harmonic functions on {S}ierpi{\'n}ski gasket type fractals.
\newblock {\em Potential Anal. 20}, 2 (2004), 125--148.

\bibitem{jonsson1980trace}
{\sc Jonsson, A., and Wallin, H.}
\newblock The trace to subsets of {${\bf R}\sp{n}$}\ of {B}esov spaces in the general case.
\newblock {\em Anal. Math. 6}, 3 (1980), 223--254.

\bibitem{jonsson1984function}
{\sc Jonsson, A., and Wallin, H.}
\newblock Function spaces on subsets of {${\bf R}^n$}.
\newblock {\em Math. Rep. 2}, 1 (1984).

\bibitem{kajino2020singularity}
{\sc Kajino, N., and Murugan, M.}
\newblock On singularity of energy measures for symmetric diffusions with full off-diagonal heat kernel estimates.
\newblock {\em Ann. Probab. 48}, 6 (2020), 2920--2951.

\bibitem{kajino2023conformal}
{\sc Kajino, N., and Murugan, M.}
\newblock On the conformal walk dimension: quasisymmetric uniformization for symmetric diffusions.
\newblock {\em Invent. Math. 231}, 1 (2023), 263--405.

\bibitem{KajinoContraction}
{\sc Kajino, N., and Shimizu, R.}
\newblock Contraction properties and differentiability of $p$-energy forms with applications to nonlinear potential theory on self-similar sets.
\newblock {\em arXiv preprint arXiv:2404.13668\/} (2024).

\bibitem{kajino2025penergyformsfractalsrecent}
{\sc Kajino, N., and Shimizu, R.}
\newblock $p$-energy forms on fractals: recent progress.
\newblock {\em arXiv preprint arXiv:2501.09002\/} (2025).

\bibitem{keithzhong}
{\sc Keith, S., and Zhong, X.}
\newblock The {P}oincar\'e{} inequality is an open ended condition.
\newblock {\em Ann. of Math. (2) 167}, 2 (2008), 575--599.

\bibitem{AnalOnFractals}
{\sc Kigami, J.}
\newblock {\em Analysis on fractals}, vol.~143 of {\em Cambridge Tracts in Mathematics}.
\newblock Cambridge University Press, Cambridge, 2001.

\bibitem{kigami}
{\sc Kigami, J.}
\newblock {\em Conductive homogeneity of compact metric spaces and construction of {$p$}-energy}, vol.~5 of {\em Memoirs of the European Mathematical Society}.
\newblock European Mathematical Society (EMS), Berlin, 2023.

\bibitem{KinnunenKorte08}
{\sc Kinnunen, J., and Korte, R.}
\newblock Characterizations of {S}obolev inequalities on metric spaces.
\newblock {\em J. Math. Anal. Appl. 344}, 2 (2008), 1093--1104.

\bibitem{kinnunenvahakangas2019}
{\sc Kinnunen, J., Korte, R., Lehrb\"ack, J., and V\"ah\"akangas, A.~V.}
\newblock A maximal function approach to two-measure {P}oincar\'e{} inequalities.
\newblock {\em J. Geom. Anal. 29}, 2 (2019), 1763--1810.

\bibitem{kinnunen2002lebesgue}
{\sc Kinnunen, J., and Latvala, V.}
\newblock Lebesgue points for {S}obolev functions on metric spaces.
\newblock {\em Rev. Mat. Iberoamericana 18}, 3 (2002), 685--700.

\bibitem{MaxFun2021}
{\sc Kinnunen, J., Lehrb\"ack, J., and V\"ah\"akangas, A.}
\newblock {\em Maximal function methods for {S}obolev spaces}, vol.~257 of {\em Mathematical Surveys and Monographs}.
\newblock American Mathematical Society, Providence, RI, [2021] \copyright 2021.

\bibitem{Koskela11}
{\sc Koskela, P., Yang, D., and Zhou, Y.}
\newblock Pointwise characterizations of {B}esov and {T}riebel-{L}izorkin spaces and quasiconformal mappings.
\newblock {\em Adv. Math. 226}, 4 (2011), 3579--3621.

\bibitem{Kumagai2014}
{\sc Kumagai, T.}
\newblock Anomalous random walks and diffusions: from fractals to random media.
\newblock In {\em Proceedings of the {I}nternational {C}ongress of {M}athematicians---{S}eoul 2014. {V}ol. {IV}\/} (2014), Kyung Moon Sa, Seoul, pp.~75--94.

\bibitem{kusuoka1992dirichlet}
{\sc Kusuoka, S., and Yin, Z.~X.}
\newblock Dirichlet forms on fractals: {P}oincar\'e{} constant and resistance.
\newblock {\em Probab. Theory Related Fields 93}, 2 (1992), 169--196.

\bibitem{makalainen2009adams}
{\sc M\"ak\"al\"ainen, T.}
\newblock Adams inequality on metric measure spaces.
\newblock {\em Rev. Mat. Iberoam. 25}, 2 (2009), 533--558.

\bibitem{Mazya85}
{\sc Maz'ya, V.~G.}
\newblock {\em Sobolev spaces}.
\newblock Springer Series in Soviet Mathematics. Springer-Verlag, Berlin, 1985.
\newblock Translated from the Russian by T. O. Shaposhnikova.

\bibitem{murugan2023note}
{\sc Murugan, M.}
\newblock A note on heat kernel estimates, resistance bounds and {P}oincar\'e{} inequality.
\newblock {\em Asian J. Math. 27}, 6 (2023), 853--866.

\bibitem{murugan2024heat}
{\sc Murugan, M.}
\newblock Heat kernel for reflected diffusion and extension property on uniform domains.
\newblock {\em Probab. Theory Related Fields 190}, 1-2 (2024), 543--599.

\bibitem{murugan2023first}
{\sc Murugan, M., and Shimizu, R.}
\newblock First-order {S}obolev spaces, self-similar energies and energy measures on the {S}ierpi\'nski carpet.
\newblock {\em Comm. Pure Appl. Math. 78}, 9 (2025), 1523--1608.

\bibitem{Perez2018DegeneratePI}
{\sc P\'erez, C., and Rela, E.}
\newblock Degenerate {P}oincar\'e-{S}obolev inequalities.
\newblock {\em Trans. Amer. Math. Soc. 372}, 9 (2019), 6087--6133.

\bibitem{saloff2002aspects}
{\sc Saloff-Coste, L.}
\newblock {\em Aspects of {S}obolev-type inequalities}, vol.~289 of {\em London Mathematical Society Lecture Note Series}.
\newblock Cambridge University Press, Cambridge, 2002.

\bibitem{SaloffCoste2011NeumannAD}
{\sc Saloff‐Coste, L., and Gyrya, P.}
\newblock Neumann and {D}irichlet heat kernels in inner uniform domains.
\newblock {\em Ast{\'e}risque\/} (2011).

\bibitem{Sasaya2025}
{\sc Sasaya, K.}
\newblock Construction of $p$-energy measures associated with strongly local $p$-energy forms.
\newblock {\em arXiv preprint arXiv:2502.10369\/} (2025).

\bibitem{Shimizu}
{\sc Shimizu, R.}
\newblock Construction of {$p$}-energy and associated energy measures on {S}ierpi\'nski carpets.
\newblock {\em Trans. Amer. Math. Soc. 377}, 2 (2024), 951--1032.

\bibitem{steinhurst2010diffusions}
{\sc Steinhurst, B.}
\newblock {\em Diffusions and {L}aplacians on {L}aakso, {B}arlow-{E}vans, and other fractals}.
\newblock ProQuest LLC, Ann Arbor, MI, 2010.
\newblock Thesis (Ph.D.)--University of Connecticut.

\bibitem{yang2025energy}
{\sc Yang, M.}
\newblock Energy inequalities for cutoff functions of $ p $-energies on metric measures spaces.
\newblock {\em arXiv preprint arXiv:2507.08577\/} (2025).

\bibitem{yang2025singularity}
{\sc Yang, M.}
\newblock On singularity of $p$-energy measures on metric measure spaces.
\newblock {\em arXiv preprint arXiv:2505.12468\/} (2025).

\end{thebibliography}

\end{document}